\documentclass[10pt, final, journal, twoside, web, letterpaper]{ieeecolor}

\usepackage{generic-no-logo}
\usepackage{flushend}
\usepackage{algorithmic}
\usepackage{graphicx}
\usepackage{textcomp}
\def\BibTeX{{\rm B\kern-.05em{\sc i\kern-.025em b}\kern-.08em
    T\kern-.1667em\lower.7ex\hbox{E}\kern-.125emX}}
\usepackage{amsmath}
\usepackage{amssymb}
\usepackage{mathtools}
\usepackage{multicol}

\usepackage{nicematrix}
\usepackage[utf8]{inputenc} 
\usepackage[T1]{fontenc}    

\usepackage[nocompress, noadjust]{cite}

\makeatletter
\let\NAT@parse\undefined
\makeatother
\usepackage{hyperref}
\hypersetup{colorlinks=true, linkcolor=blue, citecolor=blue}
\makeatletter
\renewcommand*{\eqref}[1]{%
  \hyperref[{#1}]{\textup{\tagform@{\ref*{#1}}}}%
}
\makeatother

\usepackage[export]{adjustbox}

\usepackage{url}            
\usepackage{amsfonts}       
\usepackage{nicefrac}       
\usepackage{microtype}      
\usepackage[font={footnotesize, sf}]{caption}
\usepackage{color}
\usepackage{subfigure}
\usepackage{orcidlink}

\newcommand{\Cc}{\mathcal{C}}

\newcommand{\Fc}{\mathcal{F}}
\newcommand{\Gc}{\mathcal{G}}

\newcommand{\Jc}{\mathcal{J}}
\newcommand{\Kc}{\mathcal{K}}

\newcommand{\Nc}{\mathcal{N}}
\newcommand{\Oc}{\mathcal{O}}

\newcommand{\Sc}{\mathcal{S}}

\newcommand{\Uc}{\mathcal{U}}

\newcommand{\abs}[1]{\left| #1 \right|}
\newcommand{\sgn}{\textnormal{sgn}}
\newcommand{\dist}{\textnormal{dist}}

\newcommand{\Lojasiewicz}{{\L}ojasiewicz }
\newcommand{\oprocendsymbol}{\hbox{$\bullet$}}
\newcommand{\oprocend}{\relax\ifmmode\else\unskip\hfill\fi\oprocendsymbol}

\graphicspath{{epsfiles/}}

\newcommand{\real}{\ensuremath{\mathbb{R}}}
\newcommand{\extendedreal}{\ensuremath{\overline{\mathbb{R}}}}
\newcommand{\until}[1]{\{1,\dots,#1\}}
\newcommand{\ontil}[1]{\{0,\dots,#1\}}

\newcommand{\intpos}{{\mathbb{N}}}

\newcommand{\intnonneg}{{\mathbb{Z}}_{\ge 0}}

\newcommand{\norm}[1]{\lVert#1\rVert}

\newcommand{\normB}[1]{\Big\lVert#1\Big\rVert}

\newcommand{\identity}{I}

\newcommand{\KKT}{X_\textnormal{KKT}}
\newcommand{\proxnormal}{\Nc^{\textnormal{ prox}}}
\newcommand{\ones}{\mathbf{1}}

\newcommand{\ddfrac}[2]{\frac{d #1}{d #2}}

\newcommand{\pfrac}[2]{\frac{\partial #1}{\partial #2}}

\newcommand{\dom}{\textnormal{dom}}

\newcommand{\graph}{\textnormal{graph}}
\newcommand{\Dini}{D^{+}}

\usepackage{chngcntr}
\usepackage{apptools}
\AtAppendix{\counterwithin{lemma}{section}}

\newtheorem{theorem}{Theorem}[section]
\newtheorem{proposition}[theorem]{Proposition}
\newtheorem{lemma}[theorem]{Lemma}

\newtheorem{definition}[theorem]{Definition}

\newtheorem{problem}{Problem}
\newtheorem{remark}[theorem]{Remark}

\renewcommand{\theenumi}{(\roman{enumi})}

\newcommand{\map}[3]{#1:#2\rightarrow #3}

\author{Ahmed Allibhoy\textsuperscript{\large \orcidlink{0000-0001-6193-7129}}, \IEEEmembership{Member, IEEE}, and  
Jorge Cort\'es\textsuperscript{\large \orcidlink{0000-0001-9582-5184}}, \IEEEmembership{Fellow, IEEE} 
\thanks{This work was supported in part by NSF under Award
CMMI-2044900 and AFOSR under Award FA9550-19-1-0235. \textit{(Corresponding author: Ahmed
Allibhoy.)}\\
\indent Ahmed Allibhoy is with the Department of Mechanical
Engineering, University of California Riverside, Riverside, CA 92521
USA (e-mail: \href{mailto:aallibho@ucr.edu}{aallibho@ucr.edu}).\\
\indent Jorge Cort\'es is with the Department of Mechanical and Aerospace
Engineering, University of California San Diego, San Diego, CA 92122
USA (e-mail: \href{mailto:cortes@ucsd.edu}{cortes@ucsd.edu})
}}

\title{\huge Control-Barrier-Function-Based 
Design of\\ Gradient Flows for
  Constrained Nonlinear Programming}

\begin{document}
\maketitle
\begin{abstract}
This paper considers the problem of designing a continuous-time
dynamical system that solves a constrained nonlinear optimization
problem and makes the feasible set forward invariant and
asymptotically stable. The invariance of the feasible set makes the
dynamics anytime, when viewed as an algorithm, meaning it
returns a feasible solution regardless of when it is
terminated. Our approach augments the gradient flow of the
objective function with inputs defined by the constraint functions,
treats the feasible set as a safe set, and synthesizes a safe
feedback controller using techniques from the theory of control
barrier functions. The resulting closed-loop system, termed safe gradient flow,
can be viewed as a primal-dual flow, where the state corresponds
to the primal variables and the inputs correspond to the dual ones.
We provide a detailed suite of conditions based on constraint
qualification under which (both isolated and nonisolated) local
minimizers are stable with respect to the feasible
set and the whole state space. Comparisons with other
continuous-time methods for optimization in a simple example
illustrate the advantages of the safe gradient flow.
\end{abstract}


\section{Introduction}\label{sec:intro}
\IEEEPARstart{O}{ptimization} problems are ubiquitous in engineering and applied
science.  The traditional emphasis on the numerical analysis of
algorithms is motivated by the implementation on digital platforms.
The alternate viewpoint of optimization algorithms as continuous-time
dynamical systems taken here also has a long history, often as a
precursor of the synthesis of discrete-time algorithms. This viewpoint
has been fruitful for gaining insight into qualitative properties such
as stability and convergence.

For constrained optimization problems, the picture is complicated by
the fact that algorithms may need to ensure convergence to the
optimizer as well as enforce feasibility of the iterates.  The latter
is important in real-time applications, when feasibility guarantees
may be required at all times in case the algorithm is terminated
before completion, or when the algorithm is implemented on a physical
plant where the constraints encode its safe operation.  In this paper,
we show that, just as unconstrained optimization algorithms can be
viewed as dynamical systems, constrained optimization algorithms can
be viewed as control systems. Within this framework, the task of
designing an optimization algorithm for a constrained problem is
equivalent to that of designing a feedback controller for a nonlinear
system.  We use this connection to derive a novel control-theoretic
algorithm for solving constrained nonlinear programs that combines
continuous-time gradient flows to optimize the objective function with
techniques from control barrier functions to maintain invariance of
the feasible set.

\subsection{Related Work}
Dynamical systems and optimization are closely intertwined
disciplines~\cite{KA-LH-HU:58,RWB:91,UH-JBM:94}.  The
work~\cite{AH-SB-GH-FD:21} provides a contemporary review of the
dynamical systems approach to optimization for both constrained an
unconstrained problems, with an emphasis on applications where the
optimization problem is in a feedback loop with a plant, see
e.g.~\cite{AJ-ML-PJB:09,MC-EDA-AB:20,LSPL-JWSP-EM:21}. 
Examples of such scenarios are numerous, including power
systems~\cite{MC-EDA-AB:20,LSPL-JWSP-EM:21}, network
congestion~\cite{SHL-FP-JCD:02}, and
transportation~\cite{GB-JC-JIP-EDA:22-tcns}.

\subsubsection{Flows for Equality Constrained Problems} 
For problems involving only equality constraints,~\cite{KT:80, HY:80}
employ differential geometric techniques to design a vector field that
maintains feasibility along the flow, makes the constraint set
asymptotically stable, and whose solutions converge to critical points
of the objective function.  The work~\cite{JS-IS:00} introduces a
generalized form of this vector field to deal with inequality
constraints in the form of a differential algebraic equation
and explores links with sequential quadratic
programming.

\subsubsection{Projected Gradient Methods}
Another approach to solving nonlinear programs in continuous time
makes use of projected dynamical systems~\cite{AN-DZ:96} by projecting
the gradient of the objective function onto the cone of feasible
descent directions, see e.g.,~\cite{AH-SB-FD:21}. However,
projected dynamical systems are, in general, discontinuous, which from
an analysis viewpoint requires properly dealing with notions and
existence of solutions, cf.~\cite{JC:08-csm-yo}. The
work~\cite{TLF-DHB-NJM-RLT-SG:94} proposes a continuous modification
of the projected gradient method, whose stability is analyzed
in~\cite{YSX-JW:00}.  However, this method projects onto the
constraint set itself, rather than the tangent cone, and may fail when
it is nonconvex.  Another modification is the ``constrained gradient
flow'' proposed in \cite{MM-MIJ:22}, derived using insights from
nonsmooth mechanics, and is well-defined outside the feasible set. The
resulting method is related to the one presented here 
and converges to critical points, though the
dynamics are once again discontinuous, and stability guarantees are
only provided in the case of convexity, which we do not assume.

\subsubsection{Saddle-Point Dynamics}
Convex optimization problems can be solved by searching for saddle
points of the associated Lagrangian via a primal-dual dynamics
consisting of a gradient descent in the primal variable and a gradient
ascent in the dual one. The analysis of stability and convergence of
this method has a long history~\cite{KA-LH-HU:58,TK:56}, with more
recent accounts provided for discrete-time
implementations~\cite{TL-CJ-MIJ:20} and continuous-time
ones~\cite{DF-FP:10,AC-BG-JC:17-sicon, AC-EM-SHL-JC:18-tac}.  These 
methods are particularly well suited for distributed implementation 
on network optimization problems, but they do not leave the feasible 
set invariant. 
%
%

\subsection{Contributions}
We consider the synthesis of continuous-time dynamical systems that
solve constrained optimization problems while making the feasible set
forward invariant and asymptotically stable. Our first contribution is
the design of the safe gradient flow for constrained optimization
using the framework of safety-critical control.  The basic intuition
is to combine the standard gradient flow to optimize the objective
function with the idea of keeping the feasible set safe. To maintain
safety, we augment the gradient dynamics with inputs associated with
the constraint functions and use a control barrier function approach
to design an optimization-based feedback controller that ensures
forward invariance and asymptotic stability of the feasible set. The
approach is primal-dual, in the sense that the states correspond to
the primal variables and the inputs correspond to the dual variables.

Our second contribution unveils the connection of the proposed
dynamics with the projected gradient flow.  Specifically, we provide
an alternate derivation of the safe gradient flow as a continuous
modification of the projected gradient flow, based on a design
parameter. We show that, as the parameter grows to $\infty$, the safe
gradient flow becomes the projected gradient flow. 

In addition to
establishing an interesting parallelism, we build on this equivalence
in our third contribution for understanding the regularity and
stability properties of the safe gradient flow.  We show that the flow
is locally Lipschitz (ensuring the existence and uniqueness of
classical solutions), well defined on an open set containing the
feasibility region (which allows for the possibility of infeasible
initial conditions), that its equilibria exactly correspond to the
critical points of the optimization problem, and that the objective
function is monotonically decreasing along the feasible set of the
optimization problem.  Lastly, we prove that the feasible set is
forward invariant and asymptotically stable.  

Our fourth contribution
consists of a thorough stability analysis of the critical points of
the optimization problem under the safe gradient flow.  We provide a
suite of constraint qualification-based conditions under which
isolated local minimizers are either locally asymptotically stable
with respect to the feasible set, locally asymptotically stable with
respect to the global state space, or locally exponentially stable.
We also provide conditions for semistability of nonisolated local
minimizers and establish global convergence to critical points of the
optimization problem.  Our technical analysis for this builds on a
combination of the Kurdyaka-\Lojasiewicz inequality with a novel
angle-condition Lyapunov test to establish the finite arclength of
trajectories, which we present in the appendix. 

A preliminary version of this work appeared previously
as~\cite{AA-JC:21-cdc}.  The present work significantly expands the
scope of the stability analysis of isolated local minimizers under
weaker assumptions, as well as characterizes the stability of
nonisolated local minimizers, global convergence to critical points,
and highlights the advantages of the safe gradient flow over other
continuous-time methods in optimization.

\subsection{Notation}
We let $\real$ denote the set of real numbers.  For
$v, w \in \real^n$, $v \leq w$ denotes $v_i \leq w_i$
for $i \in \until{n}$. We let $\norm{v}$ denote
the Euclidean norm and
$\norm{v}_\infty = \max_{1 \leq i \leq n} \abs{v_i}$ the infinity
norm. For $y \in \real$, we denote $[y]_+ = \max\{0, y \}$, 
and $\sgn(y) = 1$ if $y > 0$, $\sgn(y) = -1$ if $y < 0$
and $\sgn(y) = 0$ if $y = 0$. We let
$\ones_m \in \real^m$ denote the vector of all ones. For a matrix
$A \in \real^{n \times m}$, we use $\rho(A)$ and $A^\dagger$ to denote
its spectral radius and its Moore-Penrose pseudoinverse, respectively.
We write $A \succeq 0$ (resp., $A \succ 0$) to denote $A$ is positive
semidefinite (resp., $A$ is positive definite).  Given a subset
$\Cc \subset \real^n$, the distance of $x \in \real^n$ to $\Cc$ is
$\dist_\Cc(x) = \inf_{y \in \Cc} \norm{x - y}$.  We let
$\overline{\Cc}$, $\text{int}(\Cc)$, and $\partial \Cc$ denote the
closure, interior, and boundary of $\Cc$, respectively.  Given
$X \subset \real^n$ and $f:X \to \real^m$, the graph of $f$ is
$\graph(f) = \{ (x, f(x)) \mid x \in X \}$. Similarly, given a
set-valued map $\Fc:X \rightrightarrows \real^m$, its graph is
$\graph(\Fc) = \{ (x, y) \mid x \in X, y \in \Fc(x) \}$.  Given
$g:\real^n \to \real$, we denote its gradient by $\nabla g$ and its
Hessian by $\nabla^2 g$. For $g:\real^n \to \real^m$,
$\pfrac{g(x)}{x}$ denotes its Jacobian.  For
$I \subset \{1, 2, \dots, m \}$, we denote by $\pfrac{g_{I}(x)}{x}$
the matrix whose rows are $\{\nabla g_i(x)^\top\}_{i \in I}$.

\section{Preliminaries}\label{sec:preliminaries}
We present notions on invariance, stability, variational analysis,
control barrier functions, and nonlinear programming. The reader
familiar with the material can safely skip the~section.

\subsection{Invariance and Stability Notions}
We recall basic definitions from the theory of ordinary differential
equations~\cite{WH-VSC:08}.  Let $F:\real^n \to \real^n$ be a locally
Lipschitz vector field and consider the dynamical system
$\dot{x} = F(x)$. Local Lipschitzness ensures that, for every initial
condition $x_0 \in \real^n$, there exists $T > 0$ and a unique
trajectory $x:[0, T] \to \real^n$ such that $x(0) = x_0$ and
$\dot{x}(t) = F(x(t))$.  If the solution exists for all $t \geq 0$,
then it is {\emph{forward complete}}. In this case, the \emph{flow map} is
defined by $\Phi_t:\real^n \to \real^n$ such that $\Phi_t(x) = x(t)$,
where $x(t)$ is the unique solution with $x(0) = x$.  The
\emph{positive limit set} of $x \in \real^n$ is
\[
\omega(x) = \bigcap_{T \geq 0} \overline{ \{ \Phi_t(x) \mid t > T \}
}.
\]
A set $\Kc \subset \real^n$ is \emph{forward invariant} if $x \in \Kc$
implies that $\Phi_t(x) \in \Kc$ for all $t \geq 0$. If $\Kc$ is
forward invariant and $x^* \in \Kc$ is an equilibrium, $x^*$ is
\emph{Lyapunov stable relative to $\Kc$} if for every open set $U$
containing~$x^*$, there exists an open set $\tilde{U}$ also
containing~$x^*$ such that for all $x\in \tilde{U} \cap \Kc$,
$\Phi_t(x) \in U \cap \Kc$ for all $t > 0$.  The equilibrium $x^*$ is
\emph{asymptotically stable relative to $\Kc$} if it is Lyapunov
stable relative to $\Kc$ and there is an open set $U$ containing~$x^*$
such that $\Phi_t(x) \to x^*$ as $t \to \infty$ for all
$x \in U \cap \Kc$.  We say $x^*$ is \emph{exponentially stable
  relative to $\Kc$} if it is asymptotically stable relative to $\Kc$
and there exists $\mu > 0$ and an open set $U$ containing~$x^*$ such
that for all $x \in U \cap \Kc$,
$
\norm{\Phi_t(x) - x^*} \leq e^{-\mu t}\norm{x - x^*}
$.
Analogous definitions of Lyapunov stability and asymptotically
stability can be made for sets, instead of individual points.

Consider a forward invariant set $\Kc$ and a set of equilibria~$\Sc $
contained in it, $\Sc \subset \Kc$. We say $x^* \in \Sc$ is
\emph{semistable relative to $\Kc$} if $x^*$ is Lyapunov stable and,
for any open set $U$ containing~$x^*$, there is an open set
$\tilde{U}$ such that for every $x \in \tilde{U} \cap \Kc$, the
trajectory starting at $x$ converges to a Lyapunov stable equilibrium
in $U \cap \Sc$. Note that if $x^*$ is an isolated equilibrium, then
semistability is equivalent to asymptotic stability. For all the
concepts introduced here, when the invariant set is unspecified, we
mean $\Kc = \real^n$.


\subsection{Variational Analysis}
We review basic notions from variational analysis
following~\cite{RTR-RJBW:98}.  The \emph{extended real line} is
$\extendedreal = \real \cup \{\pm \infty \}$.  Given
$f: \real^n \to \extendedreal$, its \emph{domain} is
$\dom(f) = \{ x \in \real^n \mid f(x) \neq \infty, -\infty \}$.  The
\emph{indicator function of $\Cc \subset \real^n$} is
$\delta_\Cc:\real^n \to \extendedreal$,
\begin{align*}
  \delta_\Cc(x) =
  \begin{cases}
    0 & \text{if } x \in \Cc ,
    \\
    \infty & \text{if } x \notin \Cc .
  \end{cases}   
\end{align*}
Note that $\dom(\delta_\Cc) = \Cc$.  For $x \in \dom(f)$ and $d \in
\real^n$, consider the following limits
\begin{subequations}
\begin{align}
  \label{eq:first-dd}
  f'(x; d) &= \lim_{(h, y) \to (0^+, x)} \frac{f(y + hd) - f(x)}{h} ,
  \\
  \label{eq:second-dd}
  f''(x; d) &= \lim_{(h, y) \to (0^+, x)} \frac{f(y + hd) - f(x) - hf'(y; d)}{h^2}. 
\end{align}
\end{subequations}
If the limit in~\eqref{eq:first-dd} (resp. \eqref{eq:second-dd})
exists, $f$ is \emph{directionally differentiable in the direction
  $d$} (resp. \emph{twice directionally differentiable in the
  direction $d$}). By definition, $f'(x; d) = \nabla f(x)^\top d$ if
$f$ is continuously differentiable at~$x$ and
$f''(x; d) = d^\top \nabla^2 f(x) d$ if~$f$ is twice continuously
differentiable at~$x$. 

{Given a dynamical system $\dot{x} = F(x)$ and a 
function $V:\real^n \to \real$, the \emph{upper-right Dini derivative} of 
$V$ along solutions of the system is
\[ \Dini_FV(x) = \limsup_{h \to 0^+}\frac{1}{h}\left[ V(\Phi_h(x)) - V(x) \right], \]
where $\Phi_h$ is the flow map of the system. If $V$ is directionally differentiable 
then $\Dini_FV(x) = V'(x; F(x))$, and if $V$ is differentiable then $\Dini_FV(x) = \nabla V(x)^\top F(x)$.}


The \emph{tangent cone} to $\Cc \subset \real^n$ at $x \in \real^n$ is
\begin{align*}
  T_\Cc(x) &= \big\{ d \in \real^n \mid \exists
  \{t^\nu\}_{\nu=1}^{\infty} \subset (0, \infty), \{ x^\nu
  \}_{\nu=1}^{\infty} \subset \Cc
  \\
  &\qquad t^\nu \to 0^+, x^\nu \to x, \frac{x^\nu - x}{t^\nu} \to d
  \text{ as } \nu \to \infty \big\}.
\end{align*} 
If $\Cc$ is an embedded submanifold of $\real^n$, then the tangent cone
coincides with the usual differential geometric notion of tangent
space.  Let $\Pi_\Cc:\real^n \rightrightarrows \overline{\Cc}$, 
with $ \Pi_\Cc(x)
= \big\{ y \in \overline{\Cc} \mid \norm{x - y} = \dist_\Cc(x) \big\}$, be
the projection map onto $\overline{\Cc}$.  The \emph{proximal normal
  cone} to $\Cc$ at $x$ is
\begin{align*}
  \proxnormal_\Cc(x) &= \big\{ d \in \real^n \mid \exists
  \{t^\nu\}_{\nu=1}^{\infty} \subset (0, \infty),
  \\
  & \qquad \{ (x^\nu, y^\nu) \}_{\nu=1}^{\infty} \subset
  \graph(\Pi_\Cc),
  \\
  & \qquad t^\nu \to 0^+, x^\nu \to x, \frac{x^\nu - y^\nu}{t^\nu} \to
  d \text{ as } \nu \to \infty \big\}.
\end{align*}

\subsection{Safety Critical Control via Control Barrier
  Functions}\label{sec:control-barrier}
We introduce here basic concepts from safety and a method for
synthesizing safe controllers using vector control barrier functions.
Our exposition here slightly
generalizes~\cite{ADA-XX-JWG-PT:17,ADA-SC-ME-GN-KS-PT:19} to set the
stage for dealing with constrained optimization problems later.
Consider a control-affine system
\begin{equation}
  \label{eq:control-affine}
  \dot{x} = F_0(x) + \sum_{i=1}^{r}u_iF_i(x) ,
\end{equation}
with locally Lipschitz vector fields $F_i:\real^n \to \real^n$, for
$i \in \ontil{r}$, and a set $\Uc \subset \real^m$ of valid control
inputs.  Let $\Cc \subset \real^n$ represent the set of states where
the system can operate safely and $u:X \to \Uc$ be a locally Lipschitz
feedback controller, with $X \subset \real^n$ a set containing
$\Cc$. The closed-loop system~\eqref{eq:control-affine} under $u$ is
\textit{safe} with respect to $\Cc$ if $\Cc$ is forward invariant
under the closed-loop system.

Feedback controllers can be certified to be safe by resorting to the
notion of control barrier function, which we here generalize for
convenience.  Let $\Cc \subset X \subset \real^n$ and $m, k \in \intnonneg$.
{A \emph{$(m, k)$-vector control barrier function} (VCBF)} of
$\Cc$ on $X$ relative to $\Uc$ is a continuously differentiable
function $\phi:\real^n \to \real^{m + k}$ such that the following
properties hold.
\begin{enumerate}
\item {The safe set can be expressed using $m$ inequality
    constraints and $k$ equality constraints:}
  \[ \begin{aligned}
      \Cc = \{x \in \real^n \mid \;&\phi_i(x) \leq 0, \;1 \leq i \leq m, \\
      &\phi_j(x) = 0, \; m + 1 \le j \leq m + k \} ;
    \end{aligned}
  \]
  \item there exists $\alpha > 0$ such that the map
  $K:\real^n \rightrightarrows \Uc$,
  \begin{align*}
    K_\alpha(x) & = \big\{ u \in \Uc \mid \\
          & \Dini_{F_0}\phi_i(x) \!+\! \sum_{\ell=1}^{r}u_\ell\Dini_{F_\ell}\phi_i(x) \!+\! \alpha
           \phi_i(x) \leq 0,
    \\
         & \Dini_{F_0}\phi_j(x) + \sum_{\ell=1}^{r}u_\ell\Dini_{F_\ell}\phi_j(x) + \alpha
           \phi_j(x) = 0,
    \\
         &  1 \leq i \leq m, \; m+1 \le j \leq m + k \big\},
  \end{align*}
  takes nonempty values for all $x \in X$. 
\end{enumerate}

In the special case where $m=1$ and $k = 0$, this definition
coincides with the usual notion of control barrier
function~\cite[Definition~2]{ADA-SC-ME-GN-KS-PT:19}, where the class
$\Kc$ function is linear, {and the Lie derivative has been replaced 
with the upper-right Dini derivative}.  {In general, the problem of finding a
  suitable VCBF $\phi$ is problem-specific: in many cases, the
  function naturally emerges from formalizing mathematically the
  safety specifications one seeks to enforce.}  The use of
vector-valued functions instead of scalar-valued ones allows to
consider a broader class of safe sets.  If $\phi$ is a VCBF and $u$ is
a feedback where $u(x) \in K_\alpha(x)$, it follows that along
solutions to \eqref{eq:control-affine},
$\ddfrac{}{t}\phi_i(x) \leq -\alpha \phi_i(x)$ for $1 \leq i \leq m$
and $\ddfrac{}{t}\phi_j(x) = -\alpha \phi_j(x)$ for
$m + 1 \le j \leq m + k$, which implies safety of $\Cc$. This is
stated formally in the next generalization of~\cite[Thm.
2]{ADA-SC-ME-GN-KS-PT:19}.


\begin{lemma}[Safe feedback control]\label{lem:safe-feedback}
  Consider the system~\eqref{eq:control-affine} with safety set $\Cc$
  and let $\phi $ be a vector control barrier function for $\Cc$ on
  $X$.  Then if MFCQ holds on $\Cc$, any feedback controller $u:X \to \Uc$ satisfying
  $u(x) \in K_\alpha(x)$ for all $x \in X$ and such that
  $ x \mapsto F_0(x) + \sum_{i = 0}^{m}u_i(x)F_i(x)$ is locally
  Lipschitz renders $\Cc$ forward invariant and asymptotically stable.
\end{lemma}

While Lemma~\ref{lem:safe-feedback} provides sufficient conditions for
feedback controller to be safe, it does not specify how to synthesize
it. A common technique~\cite{ADA-XX-JWG-PT:17} is, for each $x \in X$,
to define $u(x)$ as the minimum-norm element of ${K_\alpha(x)}$.  Here, we
pursue an alternative design of the form
\begin{equation}
  \label{eq:qp-controller}
  u(x) \in \underset{u \in K_\alpha(x)}{\text{argmin}} \Big\{
\Big \lVert \sum_{i=1}^{r}u_iF_i(x) \Big \rVert^2 \Big\}. 
\end{equation}
This design has the interpretation of finding a controller which
guarantees safety while modifying the drift term
in~\eqref{eq:control-affine} as little as possible.  In general,
Lipschitz continuity of the closed-loop dynamics under either design is not guaranteed,
cf.~\cite{BJM-MJP-ADA:15}, so additional assumptions may be needed to
establish safety via Lemma~\ref{lem:safe-feedback}.

\subsection{Optimality Conditions for Nonlinear Programming}
We present the basic background on necessary conditions for
optimality~\cite{DPB:99}.  Consider a nonlinear program of the form
\begin{equation}\label{eq:opt}
  \begin{aligned} 
    &\underset{x \in \real^n}{\text{minimize}} &&f(x) \\
    &\text{subject to} &&g(x) \leq 0 \\
    &&&h(x)  = 0,
  \end{aligned}
\end{equation}
where $f: \real^n \to \real$, $g:\real^n \to \real^m$,
$h:\real^n \to \real^k$ are continuously differentiable. Let
\begin{equation}
  \label{eq:feasible-set}
  \Cc = \{x \in \real^n \mid g(x) \leq 0, h(x) =0 \},
\end{equation}
denote its
feasible set.  Necessary conditions for optimality can be derived
provided that the feasible set satisfies appropriate constraint
qualification conditions.  Let the active constraint, constraint
violation, and inactive constraint sets be
\begin{align*}
  I_0(x) &= \{1 \leq i \leq m \mid g_i(x) = 0\} ,
  \\
  I_+(x) &= \{1 \leq i \leq m \mid g_i(x) > 0\} ,
  \\
  I_-(x) &= \{1 \leq i \leq m \mid g_i(x) < 0\},
\end{align*}
respectively. 
We say that the optimization problem~\eqref{eq:opt} satisfies
\begin{itemize}
\item the \emph{Mangasarian-Fromovitz Constraint Qualification} (MFCQ) if
  $\{ \nabla h_j(x) \}_{j=1}^{k}$ are linearly independent and there
  exists $\xi \in \real^n$ such that $\nabla h_j(x)^\top \xi = 0$ for
  all $j \in \until{k}$ and $\nabla g_i(x)^\top \xi < 0$ for all
  $i \in I_0(x)$;
\item the \emph{Extended Mangasarian-Fromovitz Constraint Qualification}
  (EMFCQ) if $\{ \nabla h_j(x) \}_{j=1}^{k}$ are linearly independent
  and there exists $\xi \in \real^n$ such that
  $\nabla h_j(x)^\top \xi = 0$ for all $j \in \until{k}$ and
  $\nabla g_i(x)^\top \xi < 0$ for all $i \in I_0(x) \cup I_+(x)$;
\item the \emph{Linear Independence Constraint Qualification} (LICQ) at $x$,
  if
  $\{ \nabla g_i(x) \}_{i \in I_0(x)} \cup \{\nabla h_j(x)
  \}_{j=1}^{k}$ are linearly independent.
\end{itemize}
%
%

We say that the \emph{constant rank condition} (CRC) holds at $x \in \real^n$ if there exists 
an open set $U$ containing $x$ such that for all $I \subset I_0(x)$ 
the rank of $\{ \nabla g_i(y) \}_{i \in I} \cup \{\nabla h_j(y) \}_{j=1}^{k}$ 
is constant for all $y \in U$. 

If $x^* \in \Cc$ is a local minimizer, and any of the constraint qualification 
conditions hold at $x^*$, then there exists $u^* \in
\real^m$ and $v^* \in \real^k$ such that the 
\emph{Karash-Kuhn-Tucker} conditions hold,
\begin{subequations}
  \label{eq:KKT_opt}
  \begin{align}
    \label{eq:KKT_opt1}
    \nabla f(x^*) + \pfrac{g(x^*)}{x}^\top u^* +
    \pfrac{h(x^*)}{x}^\top v^* &= 0,
    \\
    g(x^*) &\leq 0,
    \\
    h(x^*) &= 0 ,
    \\
    u^* &\geq 0 ,
    \\
    (u^*)^\top g(x^*) &= 0 .
  \end{align}
\end{subequations}
The pair $(u^*, v^*)$ are called \emph{Lagrange multipliers}, 
and the triple $(x^*, u^*, v^*)$ satisfying \eqref{eq:KKT_opt} is referred
to as a \emph{KKT triple}. We denote the set of KKT
points of \eqref{eq:opt} by
\begin{align*}
  \KKT = \{x^* \in \real^n \mid \exists &(u^*, v^*) \in \real^{m}
  \times \real^{k}
  \\
  &\text{ such that $(x^*, u^*, v^*)$ solves \eqref{eq:KKT_opt}} \}.
\end{align*}

\section{Problem Formulation}\label{sec:problem}
Our goal is to solve the optimization problem~\eqref{eq:opt} by
designing a dynamical system $\dot x = F(x)$ that converges to its
solutions.  The dynamics should enjoy the following properties.
\begin{enumerate}
\item trajectories should remain feasible if they start from a
  feasible point. This can be formalized by asking the feasible set
  $\Cc$, {defined in \eqref{eq:feasible-set}}, to be forward invariant;
\item trajectories that start from an infeasible point should converge
  to the set of feasible points. This can be formalized by requiring
  that $F$ is well defined on an open set containing $\Cc$, and that
  $\Cc$ as a set is asymptotically stable with respect to the
  dynamics.
\end{enumerate}
The requirement (i) ensures that, when viewed as an algorithm, the
dynamics is \emph{anytime}, meaning that it is guaranteed to return a
feasible solution regardless of when it is terminated. The requirement
(ii) ensures in particular that trajectories beginning from infeasible
initial conditions approach the feasible set and, if the solutions of
the optimization~\eqref{eq:opt} belong to the interior of the feasible
set, such trajectories enter it in finite time, never to leave it
again.  The problem is summarized next.

\begin{problem}\label{prob:main}
  Find an open set $X$ containing $\Cc$ and design a vector field $F:X
  \to \real^n$ such that the system $\dot{x} = F(x)$ satisfies the
  following properties.
  \begin{enumerate}
  \item $F$ is locally Lipschitz on $X$;
  \item $\Cc$ is forward invariant and asymptotically stable;
  \item $x^*$ is an equilibrium if and only if $x^* \in \KKT$;
  \item $x^*$ is asymptotically stable if $x^*$ is a isolated local
    minimizer.
  \end{enumerate}
\end{problem}

\section{Constrained Nonlinear Programming via Safe Gradient
  Flow}\label{safe-gradient-design}

In this section we introduce our solution to Problem~\ref{prob:main}
in the form of a dynamical system called the \emph{safe gradient
  flow}. We present two interpretations of this system: first
from the perspective of safety critical control, where we augment the
standard gradient flow with an input and design a feedback controller
using the procedure outlined in Section \ref{sec:control-barrier},
and second as an approximation of the projected gradient
flow. Interestingly, both interpretations are equivalent.

\subsection{Safe Gradient Flow via Feedback Control}
Consider the control-affine system
\begin{equation}\label{eq:augment}
  \begin{aligned}
    \dot{x} 
    &=-\nabla f(x) - \pfrac{g(x)}{x}^\top u - \pfrac{h(x)}{x}^\top v.
  \end{aligned}
\end{equation}
One can interpret this system as the standard gradient flow of~$f$
modified by a ``control action'' incorporating 
the gradients of the constraint functions.  The intuition is that the drift term takes
care of optimizing $f$ toward the minimizer, and this direction can be
modified with the input if the trajectory gets close to the boundary
of the feasible set, cf. Figure~\ref{fig:motivation}.
\begin{figure}[!t]
  \centering \subfigure[]{
  \includegraphics[width=0.48\linewidth]{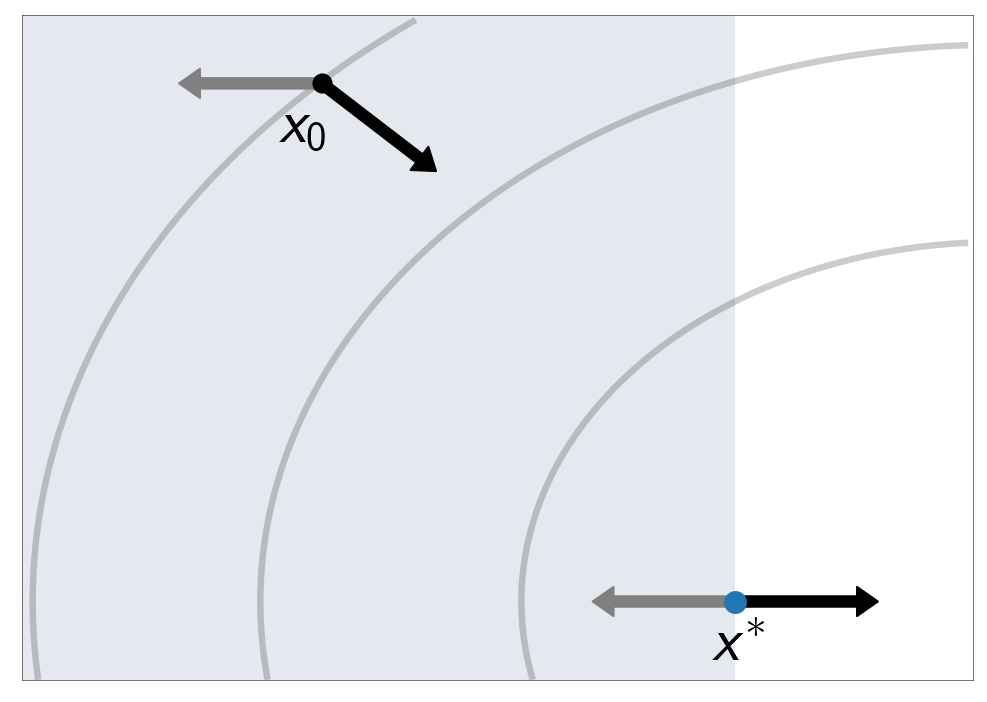}}
\subfigure[]{
  \includegraphics[width=0.48\linewidth]{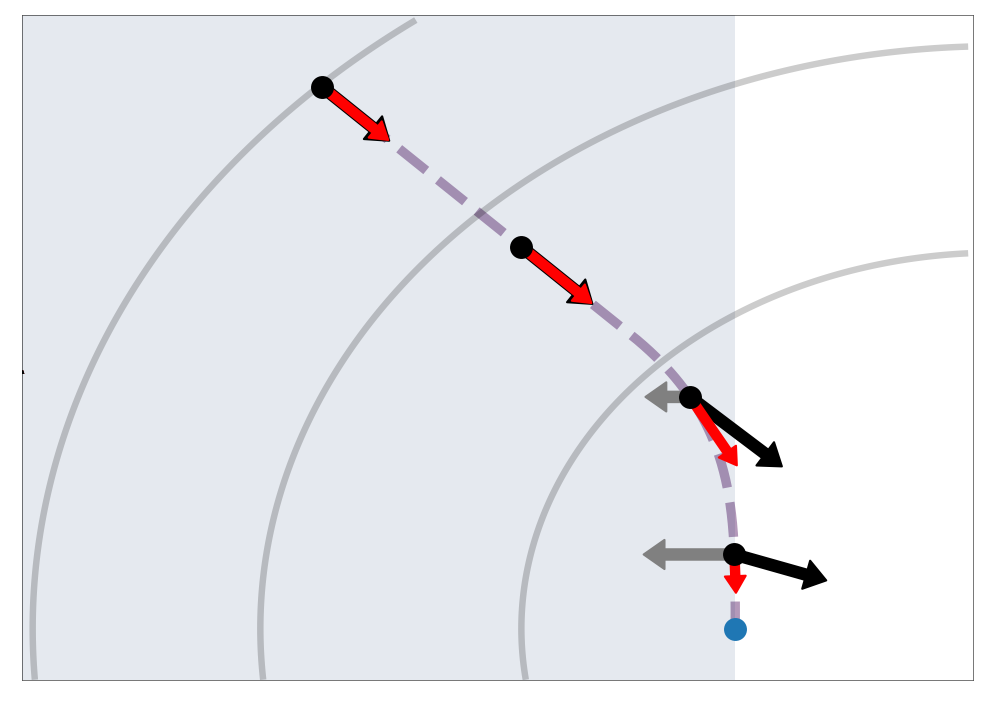}}
  \caption{Intuition behind the design of the safe gradient flow.  Grey
  lines are the level curves of the objective function and the shaded
  region is~$\Cc$.  In (a), the initial condition is $x_0$ and the
  minimizer is $x^*$, with $-\nabla f(x)$ in black and $-\nabla g(x)$
  in gray at both points. In (b), the dashed line is a trajectory
  of $\dot{x} = -\nabla f(x) - u \nabla g(x)$ starting from $x_0$. The
  black vectors are $-\nabla f(x)$, the gray vectors are
  $-u\nabla g(x)$, and the red vectors are $\dot{x}$.  Deep in the
  interior of $\Cc$, one has $u \approx 0$, as following
  the gradient of $f$ does not jeopardize feasibility while minimizing
  it. As the trajectory approaches the boundary, $u$ increases to keep
  the trajectory in~$\Cc$. }
  \label{fig:motivation}
  \vspace*{-4ex}
\end{figure}
Our idea for the controller design is to only modify the drift when
the feasibility of the state is endangered. We accomplish this by
looking at the feasible set $\Cc$ as a safe set and using 
$\map{\phi=(g, h)}{\real^n}{\real^{m + k}}$ 
as an $(m, k)$-vector control barrier function to synthesize the feedback
controller, as described next.

Let $\alpha > 0$ be a design parameter. {For reasons of space, we
  sometimes omit the arguments of functions when they are clear from
  the context}. Following Section~\ref{sec:control-barrier}, define
the admissible control set as
\begin{equation}\label{eq:Kalpha}
  \begin{aligned}
    K_\alpha&(x) = \Big\{(u, v) \in \real_{\geq 0}^m \times \real^k \, \big|
    \\
    & -\pfrac{g}{x}\pfrac{g}{x}^\top u -\pfrac{g}{x}\pfrac{h}{x}^\top v
    \leq \pfrac{g}{x}\nabla f(x) - \alpha g(x)
    \\
    & -\pfrac{h}{x}\pfrac{g}{x}^\top u - \pfrac{h}{x}\pfrac{h}{x}^\top
    v = \pfrac{h}{x}\nabla f(x) - \alpha h(x) \Big\}.
  \end{aligned}
\end{equation}
The next result shows that $\phi$ is a valid VCBF for \eqref{eq:augment}. 

\begin{lemma}[Vector control barrier function
    for~\eqref{eq:augment}]\label{lemma:feedback-feasibility}
  Consider the optimization problem~\eqref{eq:opt}. If MFCQ holds for
  all $x \in \Cc$, then there exists an open set $X$ containing $\Cc$
  such that the function $\map{\phi=(g, h)}{\real^n}{\real^{m + k}}$
  is a valid $(m, k)$-VCBF for~\eqref{eq:augment} on $X$ relative to
  $\Uc = \real^{m}_{\geq 0} \times \real^{k}$.
\end{lemma}
\begin{proof}
  We begin by showing that inequalities parameterizing
  ${K_\alpha(x)}$ are strictly feasible for all
  $x \in \Cc$, i.e., {for each $x \in \Cc$, there exists
  $\epsilon > 0$} and $(u, v) \in \real_{\geq 0}^m \times \real^k$ such
  that
  \begin{subequations}
    \label{eq:strictK}
    \begin{align}
      -\pfrac{g}{x}\pfrac{g}{x}^\top u -\pfrac{g}{x}\pfrac{h}{x}^\top v
      &\leq \pfrac{g}{x}\nabla f(x) - \alpha g(x) -\epsilon \ones_m
      \\
      \label{eq:H}
      -\pfrac{h}{x}\pfrac{g}{x}^\top u - \pfrac{h}{x}\pfrac{h}{x}^\top
      v &= \pfrac{h}{x}\nabla f(x) - \alpha h(x).
    \end{align}
  \end{subequations}
  Let $\tilde{g} = g(x) + \frac{\epsilon}{\alpha} \ones_m$.  By
  Farka's Lemma~\cite{RTR:70}, \eqref{eq:strictK} is infeasible if and
  only if there exists a solution $(u, v)$ to
  \begin{subequations}
    \label{eq:motzkin}
    \begin{align}
      \label{eq:motzkin-1}
      -\pfrac{g}{x}\pfrac{g}{x}^\top u -\pfrac{g}{x}\pfrac{h}{x}^\top v &\geq 0 \\
      \label{eq:motzkin-2}
      -\pfrac{h}{x}\pfrac{g}{x}^\top u - \pfrac{h}{x}\pfrac{h}{x}^\top v &= 0 \\
      \label{eq:motzkin-3}
      u &\geq 0 \\
      \label{eq:motzkin-t}
      u^\top \left(\pfrac{g}{x}\nabla f - \alpha \tilde{g} \right) + v^\top \left(\pfrac{h}{x}\nabla f - \alpha h(x)\right) &< 0. 
    \end{align}
  \end{subequations}
  Then \eqref{eq:motzkin-1}, \eqref{eq:motzkin-2}, and \eqref{eq:motzkin-3} imply that 
  \[
    \begin{bmatrix}
      u \\ v
    \end{bmatrix}^\top \begin{bmatrix}
      \pfrac{g}{x}\pfrac{g}{x}^\top & \pfrac{g}{x}\pfrac{h}{x}^\top \\
      \pfrac{h}{x}\pfrac{g}{x}^\top & \pfrac{h}{x}\pfrac{h}{x}^\top
    \end{bmatrix} \begin{bmatrix}
      u \\ v
    \end{bmatrix} \leq 0
  \]
  but, since the matrix is positive semidefinite,
  \begin{equation}
    \label{eq:farka1}
    (u ,v) \in \ker \begin{bmatrix}
      \pfrac{g}{x}\pfrac{g}{x}^\top & \pfrac{g}{x}\pfrac{h}{x}^\top \\
      \pfrac{h}{x}\pfrac{g}{x}^\top & \pfrac{h}{x}\pfrac{h}{x}^\top
    \end{bmatrix} = \ker \begin{bmatrix}
      \pfrac{g}{x}^\top & \pfrac{h}{x}^\top
    \end{bmatrix}.
  \end{equation}
  Next, by {$\eqref{eq:farka1}$ and that $x \in \Cc$}, \eqref{eq:motzkin-t} reduces to 
  \begin{equation}
    \label{eq:farka2}
    -u^\top(\alpha g(x) + \epsilon \ones_m) < 0,
  \end{equation}  
  and by a second application of Farka's Lemma, we see that \eqref{eq:motzkin-3}, 
  \eqref{eq:farka1} and \eqref{eq:farka2} are feasible if and only if the following 
  system is infeasible.
  \begin{subequations}
    \label{eq:farka-alternative}
    \begin{align}
      \pfrac{g(x)}{x}\xi &\leq -\alpha g(x) - \epsilon \ones_m \qquad \pfrac{h(x)}{x}\xi = 0.
    \end{align}
  \end{subequations}
  We claim that a solution to \eqref{eq:farka-alternative} can be
  constructed if MFCQ holds at $x$. Indeed, by MFCQ, there exists
  $\tilde{\xi} \in \real^n$ such that
  $\pfrac{g_{I_0}}{x}\tilde{\xi} < 0$ and
  $\pfrac{h}{x}\tilde{\xi} = 0$, and for $\epsilon$ sufficiently
  small, there exists $\gamma > 0$ such that
  $\xi = \gamma \tilde{\xi}$ solves \eqref{eq:farka-alternative}.
  Thus \eqref{eq:motzkin} is infeasible, and therefore
  \eqref{eq:strictK} is feasible.

  By strict feasibility of \eqref{eq:strictK} and 
  the fact that the matrix $\pfrac{h}{x}\pfrac{h}{x}^\top$ has full rank, 
  it can be shown by Lemma~\ref{lem:double}
  that, for all $x \in \Cc$,  
  the affine inequalities that parameterize ${K_\alpha(x)}$ are 
  \emph{regular}\footnote{Consider a linear system of inequalities of the form 
  $Cz \leq c$, $Dz = d$, 
  and a solution $z_0$. The system is regular (c.f. \cite{SMR:75}) 
  if for $C', c', D', d'$ sufficiently close to $C, c, D, d$, the 
  perturbed system $C'z \leq c'$, $D'z= d'$ remains feasible, and 
  the distance of $z_0$ to the solution set of the perturbed system is proportional to the 
  magnitude of the perturbation.}.
  Finally, since the affine inequalities parameterizing $K_\alpha$ are continuous,
  ${K_\alpha(y)}$ is nonempty
  for any $y$ sufficiently close to $x$. Hence 
  there exists an open set~$X$ such that
  $K_\alpha$ takes nonempty values on~$X$.
\end{proof}

Since $\phi$ is a VCBF, we
can design a feedback of the form~\eqref{eq:qp-controller} to maintain
safety of $\Cc$ while modifying the drift term as little as possible.
Formally,
\begin{equation}\label{eq:feedback}
  \begin{bmatrix}
    u(x) \\ v(x)
  \end{bmatrix} \in \underset{(u, v)\in K_\alpha(x)}{\text{argmin}}\left\{
    \normB{\pfrac{g(x)}{x}^\top u + \pfrac{h(x)}{x}^\top v}^2 \right\}.
\end{equation}
We refer to the closed-loop system~\eqref{eq:augment} under the
controller~\eqref{eq:feedback} as the \emph{safe gradient flow}. In
general, the solution to \eqref{eq:feedback} might not be
unique. Nevertheless, as we show later, the safe gradient flow is
well-defined because, the closed-loop behavior of the system is
independent of the chosen solution. 

Comparing \eqref{eq:augment} with the KKT equation~\eqref{eq:KKT_opt1}
suggests that $(u(x), v(x))$ can be interpreted as approximations of the
dual variables of the problem. With this interpretation, the
safe gradient flow can be viewed as a primal-dual method. We use this
viewpoint later to establish connections between the proposed method and the
projected gradient flow. 


\begin{remark}[Connection with the
    Literature]\label{rem:connection-lit}
  {\rm The work~\cite{KT:80} considers the problem of designing a
    dynamical system to solve~\eqref{eq:opt} when only equality
    constraints are present using a differential geometric
    approach. 
    Here, we show that the safe gradient
    flow generalizes the solution proposed in~\cite{KT:80}.  Under the
    assumption that $h \in C^r$ and LICQ holds, the feasible set
    $\Cc = \{x \in \real^n \mid h(x) = 0 \}$ is an embedded $C^r$
    submanifold of $\real^n$ of codimension $k$.
    The approach in~\cite{KT:80} proceeds by identifying a vector
    field $F:\real^n \to \real^n$ satisfying: (i)
    $F \in C^r$ and $F(x) \in T_\Cc(x)$ for all $x \in \Cc$; and
    (ii) $\dot{h}(x) = -\alpha h(x)$ along the trajectories of $\dot x
    = F(x)$, where $\alpha > 0$ is a design parameter.  The proposed
    vector field satisfying both properties~is
    \begin{equation}\label{eq:tanabe}
      F(x) = -\Big(I - \pfrac{h}{x}^\dagger
      \pfrac{h}{x}\Big)\nabla f(x) - \alpha \pfrac{h}{x}^\dagger
      h(x).
    \end{equation}
    To see that this corresponds to the safe gradient flow, note that
    the admissible control set~\eqref{eq:Kalpha} in this case is
    \begin{align*}
      K_\alpha(x) =\Big\{v \in \real^k \mid -\pfrac{h}{x}\nabla f(x)
      -\pfrac{h}{x}\pfrac{h}{x}^\top v = -\alpha h(x) \Big\}.
    \end{align*}
    By the LICQ assumption, ${K_\alpha(x)}$ is a singleton whose unique element
    is
    \begin{align*}
      v(x) = -\Big( \pfrac{h}{x}\pfrac{h}{x}^\top
      \Big)^{-1}\Big(\pfrac{h}{x}\nabla f(x) - \alpha h(x) \Big).
    \end{align*}
    Substituting this into~\eqref{eq:augment}, we obtain the
    expression~\eqref{eq:tanabe}.  This provides an alternative
    interpretation from a control-theoretic perspective of the
    differential-geometric design in~\cite{KT:80}, and justifies
    viewing the safe gradient flow as the natural extension to the
    case with both inequality and equality constraints.  \oprocend }
\end{remark}

\begin{remark}[Inequality Constraints via Quadratic Slack
    Variables]\label{rem:quadratic-slack}
  {\rm The work~\cite{JS-IS:00} pursues a different approach that the
    one taken here to deal with inequality constraints by reducing
    them to equality constraints. This is accomplished introducing
    quadratic slack variables. Formally, for each $i \in \until{m}$,
    one replaces the constraint $g_i(x) \leq 0$
    with the equality constraint $g_i(x) = -y_i^2$, and solves the
    equality-constrained optimization problem in the variables $(x, y)
    \in \real^{n + m}$ with a flow of the
    form~\eqref{eq:tanabe}. While this method can be expressed in
    closed form, there are several drawbacks with it. First, this
    increases the dimensionality of the problem, which can be
    problematic when there are a large number of inequality
    constraints. Second, adding quadratic slack variables introduces
    equilibrium points to the resulting flow which do not correspond
    to KKT points of the original problem.  } \oprocend
\end{remark}

\subsection{Safe Gradient Flow as an Approximation of the Projected
  Gradient Flow}

Here, we introduce an alternative design in terms of a continuous
approximation of the projected gradient flow. The latter is a
discontinuous dynamical system obtained by projecting the gradient of
the objective function onto the tangent cone of the feasible
set. Later, we show that this continuous approximation is in fact
equivalent to the safe gradient flow.

Let $x \in \Cc$ and suppose that MFCQ holds at $x$. Then the tangent
cone of $\Cc$ at $x$ is
\[ 
T_\Cc(x) = \Big\{ \xi \in \real^n \left| \pfrac{h(x)}{x}\xi = 0,
    \pfrac{g_{I_0}(x)}{x}\xi \leq 0\right. \Big\} .
\]
For $x \in \Cc$, let $\Pi_{T_\Cc(x)}$ be the projection onto
$T_\Cc(x)$. In general, the projection is a set-valued map, but the
fact that $T_\Cc(x)$ is closed and convex makes the projection onto
$T_\Cc(x)$ unique in this case. The projected gradient flow is
\begin{equation}
  \label{eq:proj-grad}
  \begin{aligned}
    \dot{x} &= \Pi_{T_\Cc(x)}(-\nabla{f}(x))
    \\
    &= \underset{\xi \in \real^n}{\text{argmin}} \quad 
    \frac{1}{2}\norm{\xi + \nabla f(x)}^2
    \\
    & \quad \; \text{subject to} \quad \pfrac{g_{I_0}(x)}{x}\xi \leq 0,
    \pfrac{h(x)}{x}\xi = 0.
  \end{aligned}
\end{equation}
In general, this system is discontinuous, so one must resort to
appropriate notions of solution trajectories and establish their
existence, see e.g.,~\cite{JC:08-csm-yo}. Here, we consider
Carath\'eodory solutions, which are absolutely continuous functions
that satisfy~\eqref{eq:proj-grad} almost everywhere.
When Carath\'eodory solutions exist in $\Cc$, then the KKT 
points of~\eqref{eq:opt} are equilibria
of~\eqref{eq:proj-grad}, and isolated local minimizers are asymptotically 
stable. 

Consider the following continuous approximation
of~\eqref{eq:proj-grad} by letting $\alpha > 0$ and defining
$\Gc_{\alpha}$ by
\begin{equation}
  \label{eq:smooth-proj-grad}
  \begin{aligned} 
    \Gc_{\alpha}(x) = \;&\underset{\xi \in \real^n}{\text{argmin}}
    &&\frac{1}{2}\norm{\xi + \nabla f(x)}^2
    \\
    &\text{subject to} && \pfrac{g(x)}{x}\xi \leq -\alpha g(x) \\
    &&& \pfrac{h(x)}{x}\xi = -\alpha h(x) .
  \end{aligned}
\end{equation}
Note that \eqref{eq:smooth-proj-grad} has a similar form
to~\eqref{eq:proj-grad}, and has a unique solution if one exists.
However, as we show later, unlike the projected gradient
flow, the vector field $\Gc_{\alpha}$ is well defined outside $ \Cc$ and
is Lipschitz.

We now show that $\Gc_{\alpha}$ approximates the projected gradient flow.
Intuitively, this is because for
$j \notin I_0(x)$, one has $g_j(x) < 0$ and hence the $j$th inequality
constraint in~\eqref{eq:smooth-proj-grad},
$\nabla g_j(x)^\top \xi \leq -\alpha g_j(x)$, becomes
$\nabla g_j(x)^\top \xi \leq \infty$ as $\alpha \to \infty$ and the
constraint is effectively removed, reducing the problem
to~\eqref{eq:proj-grad}. This is formalized next.

\begin{proposition}[$\Gc_{\alpha}$ approximates the projected
  gradient]\label{prop:approx}
  Let $x \in \Cc$ and suppose MFCQ holds. Then
  \begin{enumerate}
  \item $\Gc_{\alpha}(x) \in T_{\Cc}(x)$. 
  \item 
  $ \lim_{\alpha \to \infty}\Gc_{\alpha}(x) = \Pi_{T_\Cc(x)}(-\nabla
    f(x))$. 
  \end{enumerate}
\end{proposition}
\begin{proof}
  To show (i), note that if $x \in \Cc$, then $h(x) = 0$ and
  $g_{I_0}(x) = 0$, so the constraints in \eqref{eq:smooth-proj-grad}
  imply~$\pfrac{h(x)}{x}\Gc_{\alpha}(x) = 0$ and
  $\pfrac{g_{I_0}(x)}{x}\Gc_{\alpha}(x) \leq 0$, and therefore
  $\Gc_{\alpha}(x) \in T_{\Cc}(x)$.  
  
  Regarding (ii), for fixed $x \in
  \Cc$, let $J = I_-(x)$ and consider the following quadratic program
  \begin{equation}
    \label{eq:P}
    \begin{aligned} 
      P_x(\epsilon) = \;&\underset{\xi \in
        \real^n}{\textnormal{argmin}} &&\frac{1}{2}\norm{\xi +
        \nabla f(x)}^2 
      \\ 
      &\textnormal{subject to} && \pfrac{g_{I_0}(x)}{x}\xi \leq 0,
      \pfrac{h(x)}{x}\xi = 0
      \\
      &&& \epsilon\pfrac{g_J(x)}{x}\xi \leq -g_J(x). 
      \\
    \end{aligned}
  \end{equation}
  When $\epsilon = 0$, the feasible sets of \eqref{eq:P} and
  \eqref{eq:proj-grad} are the same. Since the objective functions are
  also the same, $P_x(0) = \Pi_{T_\Cc(x)}(-\nabla f(x))$.
  Furthermore, for all $\alpha > 0$,
  $P_x(\frac{1}{\alpha}) = \Gc_{\alpha}(x)$. Finally, since the QP
  defining $P_x$ has a unique solution, and satisfies the regularity
  conditions in \cite[Definition~2.1]{MJB-BD:95}, $P_x$ is continuous
  at $\epsilon = 0$ by \cite[Thm.~2.2]{MJB-BD:95}.  Hence
  \[ \lim_{\alpha \to \infty}\Gc_{\alpha}(x) = \lim_{\epsilon \to
    0^+}P_x(\epsilon) = P_x(0) = \Pi_{T_\Cc(x)}(-\nabla f(x)). \]
\end{proof}

\begin{figure}[!!t]
  \centering
    \includegraphics[width=0.6\linewidth]{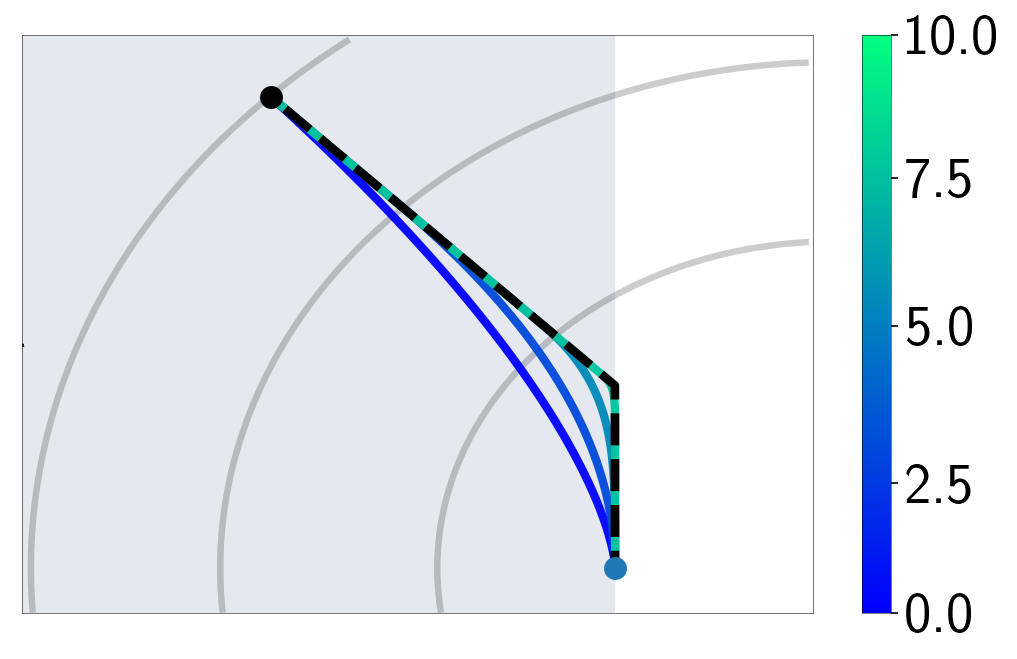}
    \caption{Projected gradient flow versus continuous
      approximation. The solution of the projected gradient flow is in
      black and solutions of $\dot{x} = \Gc_\alpha(x)$ for varying
      values of $\alpha$ are in the colors corresponding to the
      colorbar.  All solutions start from the same initial condition,
      marked by the black dot.}
  \label{fig:alphas}
  \vspace*{-2ex}
\end{figure}

A consequence of Proposition \ref{prop:approx} is that solutions of
$\dot{x} = \Gc_\alpha(x)$ approximate the solutions of the projected
gradient flow, with decreasing error as $\alpha$ increases, cf.
Figure~\ref{fig:alphas}.

\subsection{Equivalence Between the Two
  Interpretations}\label{sec:equivalence}

Here we establish the equivalence between the two interpretations of
the safe gradient flow.  Specifically, we show that the control
barrier function quadratic program~\eqref{eq:feedback} can be
interpreted as a dual program corresponding to the continuous
approximation of the projected gradient flow
in~\eqref{eq:smooth-proj-grad}.
The key to establishing the relationship between the continuous
approximation of the projected gradient flow and the safe feedback
controller in \eqref{eq:feedback} are the Lagrange multipliers of
the problem in~\eqref{eq:smooth-proj-grad}.

Let
$L:\real^n \times \real^m_{\geq 0} \times \real^k \times \real^n \to
\real$ be 
\[
\begin{aligned}
  & L(\xi, u, v; x) = \frac{1}{2}\norm{\xi + \nabla
    f(x)}^2
  \\
  & \qquad + u^\top \Big(\pfrac{g(x)}{x}\xi + \alpha g(x)\Big) +
  v^\top \Big(\pfrac{h(x)}{x}\xi +\alpha h(x)\Big).
\end{aligned}
\]
Then for each $x \in \real^n$, the Lagrangian of the optimization 
problem \eqref{eq:smooth-proj-grad}
is the function $(\xi, u, v) \mapsto L(\xi, u, v; x)$. 

For each $x\in \real^n$, the KKT conditions corresponding to 
the optimization \eqref{eq:smooth-proj-grad} are
\begin{subequations}
  \label{eq:KKT_proj}
  \begin{align}
    \label{eq:dlagrangian}
    \xi + \nabla f(x) + \pfrac{g(x)}{x}^\top u + \pfrac{h(x)}{x}^\top v &= 0 \\
    \label{eq:pinequality}
    \pfrac{g(x)}{x}\xi + \alpha g(x) &\leq 0 \\
    \label{eq:pequality}
    \pfrac{h(x)}{x}\xi + \alpha h(x) &= 0 \\
    \label{eq:nonnegativity}
    u &\geq 0 \\
    \label{eq:complementarity}
    u^\top \left(\pfrac{g(x)}{x}\xi + \alpha g(x)\right) &= 0
  \end{align}
\end{subequations}
Because the \eqref{eq:smooth-proj-grad} is strongly convex, the existence of a triple
$(\xi, u, v)$ satisfying \eqref{eq:KKT_proj} is sufficient for
optimality of $\xi$. Since the optimizer is unique, for any triple
$(\xi, u, v)$ satisfying these conditions, $\xi = \Gc_\alpha(x)$. 

Let ${\Lambda_\alpha}:\real^n \rightrightarrows \real^{m}_{\geq 0}\times \real^{k}$ be defined by
\begin{equation}
  \label{eq:lagrange-mult-set}
  \begin{aligned}
    \Lambda_\alpha(x) = \{ (u, v) \in \real^{m}_{\geq 0} \times \real^{k}
    \mid &\exists\xi \in \real^n \text{ such that } \\
    &(\xi, u, v) \text{ solves $\eqref{eq:KKT_proj}$} \}.
  \end{aligned}
\end{equation}
By definition, ${\Lambda_\alpha(x)}$ is the set of Lagrange multipliers
of~\eqref{eq:smooth-proj-grad} at $x \in \real^n$. When
${\Lambda_\alpha(x)} \neq \emptyset$, then the conditions \eqref{eq:KKT_proj}
are also necessary for optimality of \eqref{eq:smooth-proj-grad}. As
we show next, this necessity follows as a consequence of the
constraint qualification conditions.

\begin{lemma}[Necessity of optimality conditions]
  \label{lem:necessity-optimality}
  For $\alpha > 0$, if \eqref{eq:opt} satisfies MFCQ at $x \in \Cc$
  then there is an open set $U$
  containing $x$ such that $\Lambda_\alpha(x') \neq \emptyset$ for all
  $x' \in U$.
\end{lemma}
\begin{proof}
  If MFCQ holds at $x \in \Cc$, there exists
  $\xi \in \real^n$ such that $\nabla g_i(x)^\top \xi < 0$ 
  for all $i \in I_0(x)$ and $\nabla h_j(x)^\top \xi = 0$ for all $j \in \until{k}$. 
  Next, for every $j \in I_-(x)$, let $\epsilon_j > 0$ be defined as 
  \[
    \begin{aligned}
      \epsilon_j \leq
      \begin{cases}
        \frac{-\alpha g_j(x)}{\nabla g_j(x)^\top \xi} \qquad &\nabla
        g_j(x)^\top \xi > 0,
        \\
        1 \qquad &\nabla g_j(x)^\top \xi \leq 0.
    \end{cases}
  \end{aligned}
 \] 
 Then taking $0 < \epsilon \leq \min_{j \in I_-(x)} \{ \epsilon_j \}$
 and $\tilde{\xi} = \epsilon \xi$, 
 satisfies
 \begin{equation}
   \label{eq:slater}
   \pfrac{g(x)}{x} \tilde{\xi} < -\alpha g(x) \qquad \pfrac{h(x)}{x}
   \tilde{\xi} = -\alpha h(x). 
 \end{equation}
 The previous expression means that the constraints of~\eqref{eq:smooth-proj-grad}
 satisfy Slater's condition \cite[Ch.~5.2.3]{SB-LV:09} at $x$, so
 the affine constraints are \emph{regular}
 \cite[Thm. 2]{SMR:75}. This implies that there exists an open set
 $U$ containing $x$ on which~\eqref{eq:smooth-proj-grad} is feasible
 and $\Lambda_\alpha(x') \neq \emptyset$ for all $x' \in U$. 
\end{proof}

We use the optimality conditions to show that \eqref{eq:feedback} is
actually the dual problem corresponding to \eqref{eq:smooth-proj-grad}
in the appropriate sense.

\begin{proposition}[{Equivalence
      of two constructions of the safe gradient flow}]\label{prop:equiv}
  If ${\Lambda_\alpha(x)} \! \neq \! \emptyset$,
  \begin{enumerate}
  \item If $(u, v) \in \Lambda_\alpha(x)$, then $(u, v)$ solves
    \eqref{eq:feedback};
  \item ${\Gc}_\alpha$ is the closed-loop dynamics corresponding to the
    implementation of \eqref{eq:feedback} over \eqref{eq:augment}.
 \end{enumerate}
\end{proposition}
\begin{proof}
  To show (i), let $(u, v) \in \Lambda_\alpha(x)$. Then there is
  $\xi \in \real^n$ such that $(\xi, u, v)$ solves
  \eqref{eq:KKT_proj}.  By \eqref{eq:dlagrangian},
  $\xi = -\nabla f(x) - \pfrac{g(x)}{x}^\top u - \pfrac{h(x)}{x}^\top
  v$ and substituting $\xi$ into the constraints
  of~\eqref{eq:smooth-proj-grad},
  it follows immediately that $(u, v) \in K_\alpha(x)$, defined
  in~\eqref{eq:Kalpha}.  We claim that $(u, v)$ is also optimal
  for~\eqref{eq:feedback}. To prove this, let $(u', v')$ be a solution
  of \eqref{eq:feedback} and, reasoning by contradiction, suppose
  \[
  \normB{\pfrac{g(x)}{x}^\top u + \pfrac{h(x)}{x}^\top v}^2 >
  \normB{\pfrac{g(x)}{x}^\top u' + \pfrac{h(x)}{x}^\top v'}^2.  
  \]
  Then,
  $\xi' = -\nabla f(x) - \pfrac{g(x)}{x}^\top u' -
  \pfrac{h(x)}{x}^\top v'$ satisfies the constraints in
  \eqref{eq:smooth-proj-grad} and
  $\lVert {\xi' + \nabla f(x)} \rVert < \norm{\xi + \nabla f(x)}$,
  which contradicts the fact that $\xi$ is optimal
  for~\eqref{eq:smooth-proj-grad}.

  To show (ii), suppose that $(u, v)$ solves~\eqref{eq:feedback}, and
  $\xi = -\nabla f(x) - \pfrac{g(x)}{x}^\top u - \pfrac{h(x)}{x}^\top
  v$.  We claim that $\xi$ is optimal for~\eqref{eq:smooth-proj-grad}.
  Indeed, if $\tilde{\xi}$ is the optimizer
  of~\eqref{eq:smooth-proj-grad}, since ${\Lambda_\alpha(x)} \neq \emptyset$,
  there exists $(\tilde{u}, \tilde{v}) \in \Lambda_\alpha(x)$ such that
  $(\tilde{\xi}, \tilde{u}, \tilde{v})$ solves~\eqref{eq:KKT_proj}.
  Note that $(\tilde{u}, \tilde{v})$ is feasible
  for~\eqref{eq:feedback}, and
  \[ 
  \begin{aligned}
    \norm{\xi + \nabla f(x)}^2 &= \normB{\pfrac{g(x)}{x}^\top u +
      \pfrac{h(x)}{x}^\top v}^2
    \\
    &\leq \normB{\pfrac{g(x)}{x}^\top \tilde{u} + \pfrac{h(x)}{x}^\top
      \tilde{v}}^2 \\
    &= \norm{\tilde{\xi} + \nabla f(x)}^2,
  \end{aligned} 
  \]
  where the inequality follows by optimality of $(u, v)$. 
  It follows that $\xi$ is optimal, but since the optimizer
  of~\eqref{eq:smooth-proj-grad} is unique, $\xi =
  \Gc_\alpha(x)$. Hence,
  $ \Gc_\alpha(x) = -\nabla f(x) - \pfrac{g(x)}{x}^\top u -
  \pfrac{h(x)}{x}^\top v$, which is the closed-loop implementation
  of~\eqref{eq:feedback} over~\eqref{eq:augment}.
\end{proof}

\begin{remark}[Lagrange Multipliers of Continuous Approximation to
  Projected Gradient]
  {\rm The notion of duality in Proposition~\ref{prop:equiv} is weaker
    than the usual notion of Lagrangian duality.  While the result
    ensures that the Lagrange multipliers
    of~\eqref{eq:smooth-proj-grad} are solutions
    to~\eqref{eq:feedback}, the converse is not true in general.  This
    is because if $(u, v)$ solves~\eqref{eq:feedback}, then
    $({\Gc}_\alpha(x), u, v)$ might not satisfy the complementarity
    condition~\eqref{eq:complementarity}, in which case
    $(u, v) \not\in \Lambda_\alpha(x)$.  An example of this is given by the
    following
    constrained problem with objective $f$ and inequality constraints
    $g(x) \leq 0$, where
    \[
      f(x) = \norm{x}^2 \qquad g(x) = \begin{bmatrix}
        0 & 1 \\ 0 & -1
      \end{bmatrix}x - \begin{bmatrix} 1 \\ 1
      \end{bmatrix}. 
    \] 
    The constraints satisfy LICQ for all $x \in \real^n$.  The
    solution is $x^* = 0$ and $\Lambda_\alpha(x^*) = \{ (0, 0) \}$.
    However, $(1, 1)$ is an optimizer of~\eqref{eq:feedback}, even
    though $(1, 1) \notin \Lambda_\alpha(x^*)$.
    \oprocend}
\end{remark}

{
\begin{remark}[Lagrangian Dual of Continuous Approximation to
  Projected Gradient]
  {\rm
  The safe gradient flow can also be implemented using the 
  Lagrangian dual of \eqref{eq:smooth-proj-grad}. This is obtained  
  by replacing the feedback controller \eqref{eq:feedback} with 
  \begin{equation*}
    \begin{aligned}
      \begin{bmatrix} u(x) \\ v(x) \end{bmatrix} &\in 
      \underset{(u, v) \in \real^{m}_{\geq 0} \times \real^k}{\text{argmin}} \Bigg\{
      \frac{1}{2}\begin{bmatrix}
        u \\ v
      \end{bmatrix}^\top \begin{bmatrix}
        \pfrac{g}{x}\pfrac{g}{x}^\top & \pfrac{g}{x}\pfrac{h}{x}^\top \\
        \pfrac{h}{x}\pfrac{g}{x}^\top & \pfrac{h}{x}\pfrac{h}{x}^\top
      \end{bmatrix} \begin{bmatrix}
        u \\ v
      \end{bmatrix} + \\
      &  u^\top \left(\pfrac{g}{x}\nabla f - \alpha g(x)\right) + v^\top \left(\pfrac{h}{x}\nabla f - \alpha h(x)\right) \Bigg\}
    \end{aligned}
  \end{equation*}
  and considering its closed-loop implementation over \eqref{eq:augment}. Though 
  this controller no longer has the same intuitive interpretation 
  as the CBF-QP \eqref{eq:feedback}, it has the advantage that its values correspond 
  exactly with ${\Lambda_\alpha(x)}$. 
  } \oprocend
\end{remark}
}

Proposition~\ref{prop:equiv} shows that there are two equivalent
interpretations of the safe gradient flow. The first is as the
closed-loop system corresponding to~\eqref{eq:augment} with the
controller \eqref{eq:feedback}, which maintains forward invariance of
the feasible set $\Cc$ while ensuring the dynamics is as close as
possible to the gradient flow of the objective function.  The second
interpretation is as an approximation of the projection of the
gradient flow of the objective function onto the tangent cone of the
feasible set.  Both interpretations are related by the fact that the
Lagrange multipliers corresponding to the approximate projection are
the optimal control inputs solving \eqref{eq:augment}.  Beyond the
interesting theoretical parallelism, this interpretation is
instrumental in our ensuing discussion when characterizing the
equilibria, regularity, and stability properties of the safe gradient
flow.

\section{Stability Analysis of the Safe Gradient Flow}

Here we conduct a thorough analysis of the stability properties of the
safe gradient flow and show that it solves Problem~\ref{prob:main}. We
start by characterizing its equilibria and regularity properties, then
focus on establishing the stability properties of local minimizers,
and finally characterize the global convergence properties of the
flow.

\subsection{Equilibria, Regularity, and Safety}

We rely on the necessary optimality conditions introduced in
Section~\ref{sec:equivalence} to characterize the equilibria
of~${\Gc}_\alpha$.

\begin{proposition}[Equilibria of safe gradient flow
    correspond to KKT points]\label{prop:equilibria}
  If MFCQ holds at $x^* \in \Cc$, then
  \begin{enumerate}
  \item ${\Gc}_\alpha(x^*) = 0$ if and only if $x^* \in \KKT$;
  \item If $x^* \in \KKT$, then $\Lambda_\alpha(x^*)$ is the set of Lagrange
    multipliers of \eqref{eq:opt} at~$x^*$.
  \end{enumerate}
\end{proposition}
\begin{proof}
  Suppose that $\Gc_{\alpha}(x^*) = 0$. By
  Lemma~\ref{lem:necessity-optimality}, there exists
  $(u^*, v^*) \in \Lambda_\alpha(x^*)$ such that $(0, u^*, v^*)$ satisfies
  the necessary optimality conditions in \eqref{eq:KKT_proj}, which
  reduce to
  \begin{subequations}
    \label{eq:KKT_0}
    \begin{align}
      \nabla f(x^*) + \pfrac{g(x^*)}{x}^\top u^* +
      \pfrac{h(x^*)}{x}^\top v^* &= 0 
      \\
      \alpha g(x^*) &\leq 0 \\
      \alpha h(x^*) &= 0 \\
      u^* &\geq 0 \\
      (u^*)^\top (\alpha g(x^*)) &= 0
    \end{align}
  \end{subequations}
  Because $\alpha > 0$, it follows immediately that \eqref{eq:KKT_0}
  implies that $(x^*, u^*, v^*)$ satisfy \eqref{eq:KKT_opt} and $x^*
  \in \KKT$.
  
  Conversely, if $x^* \in \KKT$, then for any $(u^*, v^*)$ such that
  $(x^*, u^*, v^*)$ solves \eqref{eq:KKT_opt}, we have that
  $(0, u^*, v^*)$ solves \eqref{eq:KKT_proj}, which implies that
  ${\Gc}_\alpha(x^*) = 0$ and $(u^*, v^*) \in \Lambda_\alpha(x^*)$.
\end{proof}


Proposition~\ref{prop:equilibria}(i) shows that the safe gradient flow
meets Problem~\ref{prob:main}(iii).  The correspondence in
Proposition~\ref{prop:equilibria}(ii) between the Lagrange multipliers
of~\eqref{eq:smooth-proj-grad} and the Lagrange multipliers
of~\eqref{eq:opt} means that the proposed method can be interpreted as
a primal-dual method when implemented via~\eqref{eq:smooth-proj-grad}.
This is because the state of the system~\eqref{eq:augment} corresponds
to the primal variable of \eqref{eq:opt}, and the inputs to the
system~\eqref{eq:augment} correspond to the dual variables.

We next establish that ${\Gc}_\alpha$ is locally Lipschitz on an open set
containing $\Cc$ when the MFCQ and CRC conditions hold. This ensures the
existence and uniqueness of classical solutions to the safe gradient flow. 

\begin{proposition}[Lipschitzness of safe gradient
  flow]\label{prop:Lipschitz}
  Let $\alpha > 0$ and suppose that \eqref{eq:opt} satisfies MFCQ and CRC 
  for all $x \in \Cc$, $f, g$ and $h$ are continuously differentiable, and
  their derivatives are locally Lipschitz. Then ${\Gc}_\alpha$ is well
  defined and locally Lipschitz on an open set $X$ containing $\Cc$.
\end{proposition}
\begin{proof}
  By the proof of Lemma~\ref{lem:necessity-optimality}, if MFCQ holds
  at $x \in \Cc$, there is an open neighborhood $U_x$ containing $x$
  on which the constraints of \eqref{eq:smooth-proj-grad} satisfy
  Slater's condition. Next, 
  for all $i=1, \dots, m$ and $j=1, \dots, k$,
  \[
    \begin{aligned}
      \pfrac{}{\xi} (\nabla g_i(x)^\top \xi + \alpha g_i(x)) &=
      \nabla g_i(x)^\top ,
      \\
      \pfrac{}{\xi} (\nabla h_j(x)^\top \xi + \alpha h_j(x)) &=
      \nabla h_j(x)^\top ,
    \end{aligned}
  \]
  so the gradients of the constraints in \eqref{eq:smooth-proj-grad} are 
  the same as those in \eqref{eq:opt} and therefore  \eqref{eq:smooth-proj-grad}
  satisfies CRC. Then, ${\Gc}_\alpha$ is the unique solution to
  \eqref{eq:smooth-proj-grad} on $U_x$, and by \cite[Thm.
  3.6]{JL:95}, ${\Gc}_\alpha$ is locally Lipschitz on~$U_x$.  The desired
  result follows by letting $X=\bigcup_{x \in \Cc}U_x$.
\end{proof}

Proposition~\ref{prop:Lipschitz} verifies that the safe gradient flow
meets Problem~\ref{prob:main}(i).  Next, we show that under slightly
stronger constraint qualification conditions at KKT points, the triple
satisfying \eqref{eq:KKT_proj} is unique and Lipschitz near them.

\begin{proposition}[Lipschitzness of the solution
  to~\eqref{eq:KKT_proj}]\label{prop:strong-regularity}
  Let $x^* \in \KKT$ and suppose~\eqref{eq:opt} satisfies LICQ at
  $x^*$. Then, there exists an open set $U$ containing $x^*$ and
  Lipschitz functions $u:U \to \real_{\geq 0}^m$, $v:U\to\real^m$ such
  that $({\Gc}_\alpha(x), u(x), v(x))$ is the unique solution to
  \eqref{eq:KKT_proj} for all~$x \in U$.
\end{proposition}
\begin{proof}
  We claim that the variational equation \eqref{eq:KKT_proj} is
  \emph{strongly regular}~\cite{SMR:80} for all $x^* \in \KKT$.
  Strong regularity implies, cf.~\cite[Cor. 2.1]{SMR:80}, that
  there exists an open set $U$ containing $x^*$ and Lipschitz
  functions $\xi:U \to \real^n$, $u:U \to \real_{\geq 0}^m$,
  $v:U \to \real^k$ such that $(\xi(x), u(x), v(x))$ is the unique
  triple solving~\eqref{eq:smooth-proj-grad}. Since the solution
  \eqref{eq:smooth-proj-grad} is unique, if such a triple exists, then
  $\xi(x) = \Gc_\alpha(x)$.  To prove the claim, we begin by noting that
  \eqref{eq:smooth-proj-grad} satisfies the strong second-order
  sufficient condition since
  $ 
    \nabla^2_{\xi\xi}
    L(\xi, u, v; x) = I \succ 0.
  $
  Let $(x^*, u^*, v^*)$ be a KKT triple of~\eqref{eq:opt}.  By
  Proposition~\ref{prop:equilibria}, $(0, u^*, v^*)$ satisfies~\eqref{eq:KKT_proj}.  
  Since the $i$th inequality constraint of
  \eqref{eq:smooth-proj-grad} is
  $ \nabla g_i(x^*)^\top \xi + \alpha g_i(x^*) \leq 0$, when $\xi = 0$
  the constraint is active if and only if $g_i(x^*) = 0$.  It follows
  that when $x^* \in \KKT$, the indices of the active constraints of
  \eqref{eq:smooth-proj-grad} are the same as those of
  \eqref{eq:opt}. Moreover, by the reasoning 
  in the proof of Proposition \ref{prop:Lipschitz}, 
  gradients of the binding (i.e., the active inequality and
  equality) constraints of~\eqref{eq:smooth-proj-grad}
  and~\eqref{eq:opt} are also the same. By LICQ, the gradients of the
  binding constraints are linearly independent, which along with the
  strong second-order condition implies that \eqref{eq:KKT_proj} is
  strongly regular by \cite[Thm. 4.1]{SMR:80}.
\end{proof}

The significance of Proposition \ref{prop:strong-regularity} is
twofold.  First, it establishes that, under certain conditions, the
Lagrange multipliers of \eqref{eq:smooth-proj-grad} are Lipschitz as a
function of $x$, which ensures the existence of a locally Lipschitz
continuous feedback solving \eqref{eq:feedback}. Secondly, the result
establishes conditions for uniqueness of the Lagrange multipliers in a
neighborhood of an equilibrium~$x^*$. These facts will play an
important role in the stability analysis of local minimizers in the
sequel.

We now show in the next result that the safe gradient flow also meets
Problem~\ref{prob:main}(ii). The result follows by applying
Lemma~\ref{lem:safe-feedback} with $\phi=(g, h)$ as a VCBF and 
local Lipschitz continuity of the closed-loop dynamics, c.f.
Proposition ~\ref{prop:Lipschitz}.  

\begin{theorem}[Safety of feasible set under safe gradient
    flow]\label{thm:feasible-safety}
  Consider the optimization problem~\eqref{eq:opt}.  If MFCQ and CRC are
  satisfied on $\Cc$, then $\Cc$ is forward invariant and
  asymptotically stable under the safe gradient flow.
\end{theorem}

\begin{remark}[Advantages of safe gradient flow over
    projected gradient flow]
  {\rm Unlike the projected gradient flow, the vector field
    $\Gc_{\alpha}$ is locally Lipschitz, so classical solutions to
    $\dot{x} = \Gc_{\alpha}(x)$ exist, and the continuous-time flow can
    be numerically solved using standard ODE discretization
    schemes. Secondly, under mild conditions, $\Gc_{\alpha}$ is well
    defined for initial conditions outside $\Cc$, allowing us to
    guarantee convergence to a local minimizer starting from
    infeasible initial conditions.  Finally, because both
    \eqref{eq:proj-grad} and \eqref{eq:smooth-proj-grad} are
    least-squares problems of the same dimension subject to affine
    constraints, the computational complexity of solving either one is
    equivalent.  } \oprocend
\end{remark}

{
  \begin{remark}[Discretization of safe gradient flow and
      role of parameter~$\alpha$]
    When considering discretizations of the safe gradient flow, the
    parameter $\alpha$ plays an important role.  By construction, trajectories
    of the safe gradient flow beginning at infeasible initial
    conditions converge to~$\Cc$ at an exponential rate~$\alpha > 0$,
    so larger values of~$\alpha$ ensure faster convergence.  
    On the other hand, smaller values 
    of~$\alpha$ result in a design that enforces safety more
    conservatively and hence, intuitively, this should allow for
    larger stepsizes. Our preliminary numerical experiments with the
    forward-Euler discretization
    $x^+ = x + h\Gc_\alpha(x)$ confirm
    these intuitions, showing that larger choices of~$\alpha$ reduce
    the range of allowable stepsizes~$h$ that preserve the invariance
    of the feasible set $\Cc$ and stability of local minimizers. In
    particular, we have noticed that the maximal allowable stepsize
    $h^*_\alpha$ such that $0 < h < h^*_\alpha$ ensures stability
    and approximate safety, 
    satisfies $h^*_\alpha \to 0$ as~$\alpha \to \infty$. For
    space reasons, we leave to future work the formal characterization
    of suitable stepsizes.  \oprocend
  \end{remark}
}

\subsection{Stability of Isolated Local
  Minimizers}\label{sec:stability-isolated-local-min}

Here we characterize the stability properties of isolated local
minimizers under the safe gradient flow. The following result shows
that the safe gradient flow meets Problem~\ref{prob:main}(iv). To ensure 
existence and uniqueness of solutions, we assume throughout the rest of 
the paper that CRC holds on $\Cc$. 

\begin{theorem}[Stability of isolated local
  minimizers]\label{thm:stability}
  Consider the optimization problem~\eqref{eq:opt}. Let 
  $x^*$ be a local minimizer and an isolated KKT point, 
  and let $U$ be an open set such that $x^*$ is the only KKT point contained
  in $U$. Then,
  \begin{enumerate}
  \item If MFCQ holds for all $x \in U \cap \Cc$, then
    $x^*$ is asymptotically stable relative to $\Cc$;
    \label{item:relative-stability}
  \item If EMFCQ holds for all $x \in \bar{U}$, then $x^*$
    is asymptotically stable relative to $\real^n$;
    \label{item:aymptotic-stability}
  \item If LICQ, strict complementarity, and the second-order sufficient
    condition hold at $x^*$, then $x^*$ is exponentially stable
    relative to $\real^n$.
    \label{item:exponential-stability}
  \end{enumerate}
\end{theorem}

We divide the technical discussion leading up to the proof of the
result in three parts, corresponding to each statement.

\subsubsection{Stability of Isolated Local Minimizers Relative to
  $\Cc$}\label{sec:relative-stability}

Here we analyze the stability of local minimizers relative to the
feasible set. We start by characterizing the growth of the objective
function along solutions of the safe gradient flow.

\begin{lemma}[Growth of objective function along safe gradient
  flow]\label{lemma:decrescent}
  Let $x \in \real^n$ such that ${\Lambda_\alpha(x)} \neq \emptyset$. Then,
  for all $(u, v) \in \Lambda_\alpha(x)$,
  \[
    \Dini_{{\Gc}_\alpha}f(x) = -\norm{{\Gc}_\alpha(x)}^2 + \alpha u^\top
    g(x) + \alpha v^\top h(x).
  \]
  and if $x \in \Cc$ then, 
  \[ \Dini_{\Gc_{\alpha}}f(x) \leq 0, \] 
  with equality if and only if $x \in \KKT$.
\end{lemma}
\begin{proof}
  For $x \in X$ (with $X$ given by Proposition~\ref{prop:Lipschitz})
  such that $(u, v) \in \Lambda_\alpha(x) \neq \emptyset$,
  $({\Gc}_\alpha(x), u, v)$ solves~\eqref{eq:KKT_proj}. Next,
  \begin{align*}
    \Dini_{\Gc_{\alpha}}
    f(x) &=  \Gc_\alpha(x)^\top \nabla f(x)
    \\ 
         &\overset{(a)}{=} -{\Gc}_\alpha(x)^\top \Big({\Gc}_\alpha(x) +
           \pfrac{g(x)}{x}^\top u + \pfrac{h(x)}{x}^\top
           v\Big)
    \\ 
         &\overset{(b)}{=} -\norm{{\Gc}_\alpha(x)}^2  + \alpha u^\top g(x) + \alpha v^\top h(x), 
  \end{align*} 
  where (a) follows by rearranging \eqref{eq:dlagrangian}, 
  and (b) follows from~\eqref{eq:pequality}
  and~\eqref{eq:complementarity}. 

  To show the second statement, note that if $x\in \Cc$, 
  then $g(x) \leq 0$ and $h(x) = 0$ .
  Since $u \geq 0$, it follows~$\alpha u^\top g(x) + \alpha v^\top h(x) \leq 0$
  and therefore
  $\Dini_{\Gc_{\alpha}}f(x) \leq -\norm{{\Gc}_\alpha(x)}^2$. 
  Finally,
  $\Dini_{\Gc_{\alpha}}f(x) = 0$ if and only if ${\Gc}_\alpha(x) = 0$, which
  by Proposition~\ref{prop:equilibria}, is equivalent to $x \in \KKT$.
\end{proof}

As a consequence of Lemma~\ref{lemma:decrescent}, the objective
function decreases monotonically along the solutions starting in
$\Cc$. Thus, the objective function can be used as a Lyapunov 
function certifying asymptotic stability of an isolated equilibria 
relative to~$\Cc$.

\begin{proof}[Proof of Theorem~\ref{thm:stability}\ref{item:relative-stability}]
  By hypothesis and using Lemma~\ref{lem:necessity-optimality},
  ${\Lambda_\alpha}(x) \neq \emptyset$
  for all $x \in U \cap \Cc$. 
  Because $x^*$ is the unique strict
  minimizer of $f$ on $U \cap \Cc$, and by
  Lemma~\ref{lemma:decrescent}, $\Dini_{{\Gc}_\alpha}f(x) < 0$ for all
  $x \in U \cap\,\Cc \setminus \{x^* \}$, it follows 
  from Lemma~\ref{lem:relative-stability} that $x^*$ is asymptotically 
  stable relative to $\Cc$. 
\end{proof}

\subsubsection{Stability of Isolated Local Minimizers Relative to
  $\real^n$}

Here we establish the asymptotic stability of isolated local minima
relative to $\real^n$. To do so, we cannot rely any more on the
objective function $f$ as a Lyapunov function. This is because outside
of $\Cc$, there may exist points $x \in \real^n \setminus \Cc$ where
$f(x) < f(x^*)$. Therefore, to show stability relative to $\real^n$,
we need to identify an alternative function whose unconstrained
minimizer is~$x^*$. In fact, the problem of finding a function whose
unconstrained minimizers correspond to the local minimizers of a
nonlinear program is well studied in the optimization
literature~\cite{GdP-LG:89}: such functions are called \emph{exact
  penalty functions}. Our discussion proceeds by constructing an exact
penalty function that is also a Lyapunov function for the safe
gradient flow.

Let $\Omega \subset \real^n$ be a compact set.  A function
$V:\Omega \times (0, \infty) \to \real$ is a \emph{strong exact
  penalty function} relative to $\Omega$ if there exists
$\epsilon^* > 0$ such that for all $0 < \epsilon < \epsilon^*$,
$x^* \in \textnormal{int}(\Omega)$ is a local minimizer of
$V_\epsilon(x) := V(x, \epsilon)$ if and only if $x^*$ is a local
minimizer of \eqref{eq:opt}.  The following result gives a strong
exact penalty function for~\eqref{eq:opt}.

\begin{lemma}[Existence of strong exact penalty
    function]\label{lem:strong-exact-penalty}
  Let $\Omega \subset \real^n$ be compact such that
  $\textnormal{int}(\Omega) \cap \Cc \neq
  \emptyset$. Suppose~\eqref{eq:opt} satisfies EMFCQ at every
  $x \in \Omega$ and let  $V:\Omega \times (0, \infty) \to \real$,
  \begin{equation}
    \label{eq:exact-penalty}
    V(x, \epsilon) = f(x) + \frac{1}{\epsilon}\sum_{i=1}^{m}[g_i(x)]_+
    + \frac{1}{\epsilon}\sum_{j = 1}^{k}|h_j(x)|. 
  \end{equation}
  Then, $V$ is a strong exact penalty function relative to $\Omega$,
  $V$ is directionally differentiable on $\Omega$, and
  \[
    \begin{aligned}
      &\Dini_{{\Gc}_\alpha}V_\epsilon(x) = \Dini_{{\Gc}_\alpha}f(x) \notag
      \\
      &+ \frac{1}{\epsilon}\sum_{i \in I_+(x)}\Dini_{{\Gc}_\alpha}g_i(x)
        +\frac{1}{\epsilon}\sum_{j=1}^{k}\sgn(h_j(x))\Dini_{{\Gc}_\alpha}h_j(x), 
    \end{aligned}
  \]
  for all $x \in \Omega$.
\end{lemma}
\begin{proof}
  The fact that $V$ is a strong exact penalty function relative to
  $\Omega$ readily follows from~\cite[Thm. 4]{GdP-LG:89}.
  From~\cite[Prop. 3]{GdP-LG:89}, $V_\epsilon$ is directionally
  differentiable on $\Omega$ and its directional derivative in the
  direction~$\xi \in \real^n$~is
  \begin{align*}
    & V'_\epsilon
      (x; \xi) = \nabla f(x)^\top \xi
    \\
    &+\frac{1}{\epsilon} \sum_{i\in I_+(x)} \nabla g_i(x)^\top \xi +
      \frac{1}{\epsilon}\sum_{i \in I_0(x)}[\nabla g_i(x)^\top \xi]_+
    \\
    &+ \frac{1}{\epsilon}\sum_{
      \begin{subarray}{c}
        j \text{ such that}
        \\
        h_j(x) \neq 0
      \end{subarray}
    } \sgn(h_j(x))\nabla
    h_j(x)^\top \xi + \frac{1}{\epsilon}\sum_{
    \begin{subarray}{l}
      j \text{ such that}
      \\
      h_j(x) = 0
    \end{subarray}
    }|\nabla
    h_j(x)^\top \xi|.
  \end{align*} 
  We examine this expression in the case
  $V'_\epsilon(x; \Gc_\alpha(x)) = \Dini_{{\Gc}_\alpha}V_\epsilon(x)$.  For
  any $1 \leq i \leq m$, the definition of ${\Gc}_\alpha$ implies
  \[
    \nabla g_i(x)^\top \Gc_\alpha(x) = \Dini_{{\Gc}_\alpha}g_i(x) \leq -\alpha
    g_i(x) .
  \]
  Therefore, if $i \in I_0(x)$, then
  $[\nabla g_i(x)^\top \Gc_\alpha(x)]_+ = 0$. Similarly, for any
  $1 \leq j \leq k$, the definition of ${\Gc}_\alpha$ implies that
  \[ 
    \nabla h_j(x)^\top \Gc_\alpha(x) = \Dini_{{\Gc}_\alpha} h_j(x) = -\alpha
  h_j(x),
  \] 
  so if $h_j(x) = 0$, then
  $|\nabla h_j(x)^\top \Gc_\alpha(x)| = 0$, and the result follows.
\end{proof}

We now show that $V_\epsilon$ is a Lyapunov function for $\epsilon$
sufficiently small and use this fact to certify the asymptotic
stability of isolated local minimizers.

\begin{proof}[Proof of
  Theorem~\ref{thm:stability}\ref{item:aymptotic-stability}]
  Assume, without loss of generality, that $U$ 
  is bounded. By Lemma~\ref{lem:strong-exact-penalty}, the function $V_\epsilon$ defined
  in~\eqref{eq:exact-penalty} is a strong exact penalty relative
  to~$\overline{U}$. By definition, this means that there exists
  $\epsilon_1 > 0$ such that when $\epsilon < \epsilon_1$, $x^*$ is
  the only minimizer of $V_\epsilon$ in $U$.  Let $x \in U$ and
  $(u, v) \in \Lambda_\alpha(x)$. Then, using Lemmas~\ref{lemma:decrescent}
  and~\ref{lem:strong-exact-penalty} and the definition of ${\Gc}_\alpha$,
  we have
  \[
    \begin{aligned}
      \Dini_{{\Gc}_\alpha}V_\epsilon(x) \leq &-\norm{{\Gc}_\alpha(x)}^2 + \alpha
      u^\top g(x) + \alpha v^\top h(x)
      \\
      &-\frac{1}{\epsilon} \sum_{i \in I_+(x)}\alpha g_i(x)
      -\frac{1}{\epsilon} \sum_{j=1}^{k}\alpha |h_j(x)|.
    \end{aligned}
  \]
  Let $I_{\leq 0} = I_0(x) \cup I_-(x)$. It follows that,
  \[
  \begin{aligned}
    &\Dini_{{\Gc}_\alpha}V_\epsilon(x) \leq -\norm{{\Gc}_\alpha(x)}^2 +
    \alpha\sum_{i \in I_{\leq 0}}u_ig_i(x)\;+
    \\
    &\alpha\sum_{i \in I_+(x)}\Big( u_i - \frac{1}{\epsilon} \Big)g_i(x)
    + \alpha\sum_{j=1}^{k}\Big(|v_j| -
    \frac{1}{\epsilon}\Big)|h_j(x)|.
  \end{aligned}
  \] 
  Next, choose $0 < \epsilon_2 < \frac{1}{B}$ 
  where $B > 0$ satisfies the bound given by 
  Lemma~\ref{lem:uniform-boundedness}.
  Then, for $\epsilon < \epsilon_2$,
  \[
  \begin{aligned}
    &\sum_{i \in I_+(x)}\Big( u_i - \frac{1}{\epsilon}
    \Big)g_i(x)+\sum_{j=1}^{k}\Big(|v_j| -
    \frac{1}{\epsilon}\Big)|h_j(x)| < 0.
  \end{aligned} 
  \]
  Finally, since $u \geq 0$, we have $\alpha \sum_{i \in I_{\leq 0}}u_ig_i(x) \leq 0$. Thus,
  \[ \Dini_{{\Gc}_\alpha}V_\epsilon(x) \leq -\norm{{\Gc}_\alpha(x)}^2 < 0, \]
  for all $x \in U \setminus \{x^*\}$, whenever 
  $\epsilon < \min\{\epsilon_1, \epsilon_2 \}$. Therefore $V_\epsilon$ is a
  Lyapunov function on $U$ and asymptotic stability of $x^*$ 
  relative to $\real^n$ follows
  by Lemma~\ref{lem:relative-stability}.
\end{proof}

\begin{remark}[Relationship to merit functions in numerical
  optimization]
{\rm In numerical optimization, the $\ell^1$ penalty function
  in~\eqref{eq:exact-penalty} is often used as a \emph{merit
    function}, i.e., a function that quantifies how well a single
  iteration of an optimization algorithm balances the two goals of
  reducing the value of the objective function and reducing the
  constraint violation (cf. \cite[Sec. 15.4]{JN-SW:06}). Typically,
  the stepsize on each iteration is chosen so that the merit function
  is nonincreasing. Thus, if the algorithm is viewed as a
  discrete-time dynamical system, the merit function is a Lyapunov
  function.  The $\ell^1$ penalty plays a similar role for the
  continuous-time system described here.}  \oprocend
\end{remark}

\subsubsection{Exponential Stability of Isolated Local Minimizers}

We now discuss the exponentially stability of isolated local
minimizers. Our first step is to identify conditions under which the
safe gradient flow is differentiable. To do so, we introduce the
notions of strict complementarity and second-order condition on the
optimization problem.

\begin{definition}[Strict complementarity and second-order sufficient
  conditions]
  Let $(x^*, u^*, v^*)$ be a KKT triple of~\eqref{eq:opt}.
  \begin{itemize}
  \item The strict complementarity condition holds if $u^*_i > 0$ for
    all $i \in I_0(x^*)$;
  \item The second-order sufficient condition holds if $z^\top Q z> 0$
    for all
    $z \in \ker \pfrac{g_{I_0}(x^*)}{x} \cap \ker \pfrac{h(x^*)}{x}$, where
    \begin{equation}
      \label{eq:Q}
      Q \!=\! \nabla^2 f(x^*) \!+\! \sum_{i = 1}^{m}u_i^* \nabla^2 g_i(x^*) \!+\!
      \sum_{j =1 }^{k}v_i^* \nabla^2 h_j(x^*). 
    \end{equation}
  \end{itemize}
\end{definition}

When LICQ holds, strict complementarity together with the second-order
sufficient condition can be used to establish the differentiability of
a KKT triple of a nonlinear parametric program with respect to the
parameters \cite{AVF:76}. When these conditions are satisfied
by~\eqref{eq:opt}, we show next that $\Gc_\alpha$ is
differentiable and provide an expression for its Jacobian. 
A step-by-step computation of the Jacobian is in Appendix~\ref{ap:jacobian}. 

\begin{lemma}[Jacobian of safe gradient flow]
  \label{lem:jacobian}
  Let $x^* \in \KKT$ and $(u^*, v^*)$ be the associated Lagrange
  multipliers. Suppose
  \begin{itemize}
    \item LICQ holds at $x^*$;
    \item $(x^*, u^*, v^*)$ satisfies the strict complementarity condition;
    \item $(x^*, u^*, v^*)$ satisfies the second-order sufficient condition.
  \end{itemize}
  Then $\Gc_{\alpha}$ is differentiable at $x^*$ and 
  \begin{equation}
    \label{eq:jac-G}
    \pfrac{{\Gc}_\alpha(x^*)}{x} = -PQ - \alpha(\identity - P),
  \end{equation}
  \vspace{-0.2ex}where $\identity$ is the $n \times n$ identity matrix, $P$ is the orthogonal
  projection matrix onto
  $\ker \pfrac{g_{I_0}(x^*)}{x} \cap \ker \pfrac{h(x^*)}{x}$, and $Q$
  is defined in~\eqref{eq:Q}.
\end{lemma}
\begin{proof}
  By Proposition~\ref{prop:strong-regularity}, there is a neighborhood
  $U$ of~$x^*$ where the unique KKT triple
  of~\eqref{eq:smooth-proj-grad} corresponding to~$x \in U$
  is~$({\Gc}_\alpha(x), u(x), v(x))$.  From the proof of that result, LICQ
  and the second-order sufficient condition hold for
  \eqref{eq:smooth-proj-grad} at $x^*$.  Further, the indices of the
  active constraints of \eqref{eq:opt} are the same as those of
  \eqref{eq:smooth-proj-grad}. Because $(x^*, u^*, v^*)$ satisfies the
  strict complementarity condition for~\eqref{eq:opt}, $(0, u^*, v^*)$
  satisfies the strict complementarity condition
  for~\eqref{eq:smooth-proj-grad}.  Thus, by~\cite[Cor.
  1]{KJ:84}, ${\Gc}_\alpha$ is continuously differentiable at $x^*$, 
  and, by following the steps in Appendix \ref{ap:jacobian}, we obtain 
  \eqref{eq:jac-G}. 
\end{proof}

Using the result in Lemma~\ref{lem:jacobian}, stability of an isolated
local minimizer can be inferred by showing that the eigenvalues of the
Jacobian of the safe gradient flow are all strictly negative.

\begin{proof}[Proof of
  Theorem~\ref{thm:stability}\ref{item:exponential-stability}]
  By the second-order sufficient condition, $z^\top PQP z>0$ for all
  $z \in \text{im }P \setminus \{0 \}$. It follows that 
  $PQPz = 0$ if and only if
  $z \in \ker P$.  Therefore $0$ is an eigenvalue of $PQP$ with
  multiplicity $r$ and $PQP$ has $n - r$ strictly positive
  eigenvalues, where
    $r = \dim \ker P$.
  Let $z_1, \dots, z_r$ be the eigenvectors corresponding to the zero
  eigenvalues, and $z_{r + 1}, \dots, z_n$ be eigenvectors
  corresponding to the positive eigenvalues, denoted
  $\lambda_{r + 1}, \dots, \lambda_n$.  Then
  \[
    \begin{aligned}
      Pz_i = \begin{cases} 0 \qquad &i=1, \dots, r ,
        \\
        z_i \qquad &i=r+1, \dots, n .
      \end{cases}
    \end{aligned}
  \]
  Let
  \[
    \mu_i =
    \begin{aligned}
      \begin{cases}
        0 \qquad &i=1, \dots, r,
        \\
        \lambda_i - \alpha \qquad &i=r + 1, \dots, n .
      \end{cases}\end{aligned}
  \]
  Then, it follows that $(PQP - \alpha P)z_i = \mu_i z_i$ for all
  $1 \leq i \leq n$. Observe that
  $PQP - \alpha P = (PQ - \alpha \identity)P$ has precisely the same
  eigenvalues as $P(PQ - \alpha \identity) = PQ - \alpha P$.
  Therefore, since $\mu_i$ is an eigenvalue of $PQ - \alpha P$, it
  follows that $\mu_i + \alpha$ is an eigenvalue of
  \[
    PQ - \alpha P + \alpha \identity = PQ + \alpha (\identity - P) =
    -\pfrac{{\Gc}_\alpha(x^*)}{x}.
  \]
  Hence the eigenvalues of $\pfrac{{\Gc}_\alpha(x^*)}{x}$ are
  \[
    \{ -\alpha, -\alpha, \dots, -\alpha, -\lambda_{r + 1}, \dots,
    -\lambda_{n} \},
  \]
  where $-\alpha$ appears with multiplicity $r$. Since all the
  eigenvalues are strictly negative, $x^*$ is exponentially stable.
\end{proof}

\subsection{Stability of Nonisolated Local Minimizers}

We have characterized in
Section~\ref{sec:stability-isolated-local-min} the stability under the
safe gradient flow of local minimizers that are isolated KKT
points. In general, if $x^*$ is strict local minimizer that is not an
isolated KKT point (for example, if there are an infinite number of
local maximizers arbitrarily close to $x^*$, cf. \cite[page
5]{PAA-KK:06}), or if $x^*$ is only a local minimizer, then there are
no guarantees on Lyapunov stability. However, as we show here,
nonisolated minimizers are stable under the safe gradient flow under
additional assumptions on the problem data.

When there are no constraints, the safe gradient flow reduces to the
classical gradient flow, where conditions for semistability of local
minimizers are well known: if the objective function is real-analytic,
then all trajectories of the gradient flow have finite arclength,
cf.~\cite{SJ:83}, in which case the objective function can be used to
construct an arclength-based Lyapunov function satisfying the
hypotheses of Lemma~\ref{lem:cont-stability} to establish
semistability. In this section, we conduct a similar analysis for the
constrained case. Our main result is as follows.

\begin{theorem}[Stability of nonisolated local
    minima]\label{thm:non-isolated}
  Consider the optimization problem~\eqref{eq:opt}, and assume $f$,
  $g$ and $h$ are real-analytic. Let $\Sc$ be a bounded set of local
  minimizers on which $f$ is constant and equal to $f^*$ such that
  \begin{enumerate}
  \item There is an open set $U$ and $\beta > 0$ such that
    $U \cap \KKT = \Sc$ and $f(x) - f^* \geq \beta\dist_\Sc(x)^2$ for
    all $x \in U \cap \Cc$;
    \label{item:weak-sharp}
  \item LICQ is satisfied at all $x^* \in \Sc$;
  \item $T_\Sc(x^*) \cap \proxnormal_\Sc(x^*) = \{ 0\}$ for all $x^* \in \Sc$. 
  \end{enumerate}
  Then there is $\alpha^* > 0$ such that every $x^* \in \Sc$ is
  semistable relative to $\real^n$ under the safe gradient flow
  ${\Gc}_\alpha$, for $\alpha > \alpha^*$.
\end{theorem}

To prove this result, we first discuss various intermediate
results. In particular, the growth condition in
Theorem~\ref{thm:non-isolated}(i) plays a crucial role in the
construction of a Lyapunov function to prove the result.  Any
$x^* \in \Sc$ satisfying this property is called a \emph{weak sharp
  minimizer of $f$ relative to $\Sc$}.  Weak sharp minimizers play an
important role in sensitivity analysis for nonlinear programs as well
as convergence analysis for numerical methods in
optimization~\cite{JVB-MCF:93,MS-DEW:99}.

We review second-order optimality conditions for weak sharp
minimizers.  Let $x^* \in \KKT$, suppose that LICQ holds at~$x^*$, and
let $(u^*, v^*)$ be the unique Lagrange multipliers of~\eqref{eq:opt}
associated to $x^*$. Define the index set of \emph{strongly active}
constraints as 
\[ I_0^+(x^*) = \{ 1 \leq i \leq m \,| \, u_i^* > 0 \}. \]
The \emph{critical cone} is
\begin{equation}
  \label{eq:descent-cone}
  \begin{aligned}
    \Gamma(x^*) & = \{ d \in \real^n \mid \nabla h_j(x^*)^\top d = 0,
    j=1, \dots k,
    \\
    & \qquad \nabla g_i(x^*)^\top d = 0, i \in I_0^+(x^*),
    \\
    & \qquad \nabla g_j(x^*)^\top d \leq 0, j \in I_0(x^*) \setminus
    I_0^+(x) \}.
  \end{aligned}
\end{equation}

\begin{lemma}[Second-order necessary condition for
    constrained weak sharp minima {\cite[Prop. 3.5]{MS-DEW:99}}] 
  \label{lem:necessary-condition}
  Consider~\eqref{eq:opt} and let $\Sc \subset \Cc$ be a set on which
  $f$ is constant.  Suppose that $x^* \in \partial \Sc$ is a weak
  sharp local minimizer of $f$ relative to $\Sc$ and LICQ is satisfied
  at $x^*$. Let $u^*, v^*$ be the Lagrange multipliers and define
  $ \ell(x) = f(x) + (u^*)^\top g(x) + (v^*)^\top h(x)$.  Then, there
  exists $\gamma > 0$ such that, for all $d \in \Gamma(x^*)$,
  \[
    \ell''(x^*; d) \geq \gamma \dist_{T_\Sc(x^*)}(d)^2.
  \]
\end{lemma}

\begin{lemma}[Second-order sufficient condition for
    unconstrained weak sharp minima {\cite[Thm. 2.5]{MS-DEW:99}}]
    \label{lem:sufficient-condition}
  Consider $W:\real^n \to \real$ and suppose 
  that $W$ is constant on $\Sc$. Suppose $x^* \in \partial \Sc$ and  
  $W''(x^*; d) > 0$ 
  for all $d \in \proxnormal_\Sc(x^*) \setminus \{ 0 \}$, 
  then $x^*$ is a weak sharp local minimizer of $W$ relative to~$\Sc$. 
\end{lemma}

We now proceed with the construction of the Lyapunov function. Let
$T^{(\alpha)}_\Cc:\real^n \rightrightarrows \real^n$ be the set-valued
map where, for each $x \in \real^n$, $T^{(\alpha)}_\Cc(x)$ is the
constraint set of \eqref{eq:smooth-proj-grad}.
Let $J_\alpha:\real^n \times \real^n \to \real$ be
\[
  J_\alpha(x, \xi) = \alpha f(x)  + \nabla f(x)^\top \xi +
  \frac{1}{2}\norm{\xi}^2.
\]
Consider the optimization problem
\begin{equation}
  \label{eq:alt-opt}
    \begin{aligned} 
      &\underset{\xi \in T^{(\alpha)}_\Cc(x)}{\text{minimize}} && J_\alpha(x, \xi) 
    \end{aligned}
\end{equation}
As we show next, {the solution to \eqref{eq:alt-opt} is
  \eqref{eq:smooth-proj-grad}.}

\begin{lemma}[Correspondence between~\eqref{eq:alt-opt}
    and~\eqref{eq:smooth-proj-grad}]
  \label{lem:alt-opt}
  Let $x \in \real^n$. Then the program \eqref{eq:smooth-proj-grad} has a solution 
  at $x$ if and only if \eqref{eq:alt-opt} has a solution, in which case
  $
    \Gc_\alpha(x) = \arg\min_{\xi \in T^{(\alpha)}_\Cc(x)}\{ 
      J_\alpha(x, \xi) \}.
  $
\end{lemma}
\begin{proof}
  Note that the feasible sets of \eqref{eq:alt-opt} and \eqref{eq:smooth-proj-grad}
  coincide. Next, for all~$(x, \xi) \in \real^n \times \real^n$
  \[
    \begin{aligned}
      J_\alpha(x, \xi) - \frac{1}{2}\norm{\xi + \nabla f(x)}^2 =
      \alpha f(x) - \frac{1}{2}\norm{\nabla f(x)}^2.
    \end{aligned}
  \]
  Since the difference of the objectives in \eqref{eq:alt-opt} and
  \eqref{eq:smooth-proj-grad} does not depend on~$\xi$, both problems
  have the same optimizer.
\end{proof}

Lemma~\ref{lem:alt-opt} shows that \eqref{eq:alt-opt} is another
characterization of the safe gradient flow in terms of a parametric
quadratic program.  Let $W_\alpha:X \to \real$ be the value function
\begin{equation}
  \label{eq:lyap-W}
  \begin{aligned}
    W_\alpha(x) &= \inf_{\xi \in T^{(\alpha)}_\Cc(x)} \{ J_\alpha(x, \xi) \}
    \\
    &= \alpha f(x) + \nabla f(x)^\top \Gc_\alpha(x) +
    \frac{1}{2}\norm{{\Gc}_\alpha(x)}^2 .
  \end{aligned}
\end{equation}
Our strategy to prove Theorem~\ref{thm:non-isolated} consists of
showing that $W_\alpha$ is a Lyapunov function satisfying the
hypotheses in Lemma~\ref{lem:angle-convergence} whenever $\alpha$ is
sufficiently large. Towards this end, we begin by computing the
directional derivative of $W_\alpha$. Let
$Q:X \times \real_{\geq 0}^m \times \real^k \to \real^{n \times n}$ be
the matrix-valued function,
\[
  Q(x, u, v) = \nabla^2 f(x) + \sum_{i= 1}^{m}u_i \nabla^2 g_i(x) +
  \sum_{j =1}^{k}v_j \nabla^2 h_j(x).
\]
Since the Lagrange multipliers, $(u(x), v(x))$ are unique in a
neighborhood of $\Sc$, we slightly abuse notation by defining
$Q(x):=Q(x, u(x), v(x))$. By Lipschitzness of $u$ and $v$, $Q$ is
continuous on $X$. The proof of the next result follows
from~\cite[Thm. 2]{KJ:84} and \cite[Cor. 4.1]{AS:85} and is
omitted for brevity.

\begin{lemma}[Differentiability of $W_\alpha$]
  \label{lem:diff-W}
  Suppose that $\Sc$ satisfies the hypotheses in
  Theorem~\ref{thm:non-isolated}, and $X$ is an open set containing
  $\Sc$ on which $({\Gc}_\alpha(x), u(x), v(x))$ is the unique solution
  to~\eqref{eq:KKT_proj}. Then
  \begin{enumerate}
  \item For all $x \in X$, $W_\alpha$ is differentiable with
    \begin{align}
      \label{eq:gradW}
      \nabla W_\alpha(x) = -(\alpha I - Q(x))\Gc_{\alpha}(x) ;
    \end{align}
  \item For all $x^* \in \Sc$, $W_\alpha$ is twice directionally
    differentiable in any direction $d \in \real^n$, where
  \end{enumerate}
  \begin{equation}
    \label{eq:directional-derivative}
    \begin{aligned} 
      W_{\alpha}''(x^*; d) =  &\underset{\zeta \in
        \real^n}{\textnormal{min}} && \begin{bmatrix}
        d \\ \zeta
      \end{bmatrix}^\top\begin{bmatrix}
        \alpha Q(x^*) & Q(x^*) \\ Q(x^*) & I
      \end{bmatrix}\begin{bmatrix}
        d \\ \zeta
      \end{bmatrix}
      \\
      &\textnormal{s.t.} && \alpha \nabla h_j(x^*)^\top d + \nabla
      h_j(x^*)^\top \zeta = 0,
      \\
      &&& \qquad \forall j=1, \dots, k ,
      \\
      &&& \alpha \nabla g_i(x^*)^\top d + \nabla g_i(x^*)^\top \zeta
      = 0,
      \\
      &&& \qquad \forall i \in I^+_0(x^*),
      \\
      &&& \alpha \nabla g_s(x^*)^\top d + \nabla g_s(x^*)^\top \zeta
      \leq 0,
      \\
      &&& \qquad \forall s \in I_0(x^*) \setminus I^+_0(x^*).
    \end{aligned}
  \end{equation}
\end{lemma}

\begin{remark}[Dependence of $Q(x)$ on $\alpha$]
  {\rm In general, given~$x \in X$, the value of $Q(x)$ depends on the
    choice of~$\alpha$, since~$(u(x), v(x))$ depends on
    $\alpha$.  However, if $x^* \in \KKT$, then~$(u(x^*), v(x^*))$
    corresponds to the Lagrange multipliers of \eqref{eq:opt} and
    $Q(x^*)$ is the Hessian of the Lagrangian of \eqref{eq:opt}.  In
    particular, this means that for $x^* \in \KKT$, the value 
    of~$Q(x^*)$ depends only on the problem data and is independent of
    $\alpha$.  } \oprocend
\end{remark}

We now proceed with the proof of Theorem~\ref{thm:non-isolated}.

\begin{proof}[Proof of Theorem~\ref{thm:non-isolated}]
  Let $\alpha^* = \sup_{x^* \in \Sc} \{ \rho(Q(x^*)) \}$.  For
  $\alpha > \alpha^*$, we have $\alpha I - Q(x^*) \succ 0$ for all
  $x^* \in \Sc$. Assume without loss of generality that
  $\alpha I - Q(x) \succ 0$ for all $x \in U$ (if not, since $Q$ is
  continuous, we can always find an open subset of $U$ containing
  $\Sc$ for which these property holds).  We claim that $W_\alpha$
  satisfies each of the conditions~(i)-(iii) in
  Lemma~\ref{lem:angle-convergence} with $\Kc = \real^n$.

  We begin by showing condition (iii).  If $x^* \in U$ is a local
  minimizer of $W_\alpha$, then
  $\nabla W_\alpha(x^*) = (\alpha I - Q(x^*))\Gc_{\alpha}(x^*) =
  0$. Since $\alpha I - Q(x^*) \succ 0$, from~\eqref{eq:gradW} we
  deduce ${\Gc}_\alpha(x^*) = 0$, so $x^* \in \KKT$ and therefore
  $x^* \in U \cap \KKT = \Sc$.
  
  Conversely, suppose that $x^* \in \Sc$.  Note that, by
  Proposition~\ref{prop:equilibria}, $W_\alpha(x) = \alpha f(x)$ for
  all $x \in \Sc$. Therefore, if $x^* \in \text{int}(\Sc)$, it follows
  that $x^*$ is a local minimizer of $W_\alpha$.  On the other hand, suppose that 
  $x^* \in \partial \Sc$.  For $d \in \real^n$, let $\zeta_d$ be the
  unique optimizer of \eqref{eq:directional-derivative}. Then
  \begin{equation}
    \label{eq:W_alpha}
    W''_\alpha(x^*; d) = \alpha d^\top Q(x^*)d + 2 \zeta_d ^\top
    Q(x^*)d + \norm{\zeta_d}^2.  
  \end{equation}
  From the constraints in \eqref{eq:directional-derivative},
  $\zeta_d + \alpha d \in \Gamma(x^*)$. Because $x^* \in \partial\Sc$
  is a weak sharp minimizer of $f$ relative to $\Sc$, by
  Lemma~\ref{lem:necessary-condition}, there exists $\gamma > 0$ such
  that
  \begin{equation}
    \label{eq:Ldd}
    \begin{aligned}
      \ell''(x^*; \zeta_d + \alpha d) &= (\zeta_d + \alpha d)^\top
      \nabla^2 \ell(x^*)(\zeta_d + \alpha d) ,
      \\
      &\geq \gamma \dist_{T_\Sc(x^*)}(\zeta_d + \alpha d)^2 , \;
      \forall d \in \real^n.
    \end{aligned}
  \end{equation}
  Since $\nabla^2 \ell(x^*) = Q(x^*)$, we combine \eqref{eq:W_alpha}
  and \eqref{eq:Ldd} to get
  \[
    \begin{aligned}
      \alpha W''_\alpha(x^*; d) \geq &\zeta_d^\top(\alpha I - Q(x^*))\zeta_d
      + \gamma \dist_{T_\Sc(x^*)}(\zeta_d + \alpha d)^2.
    \end{aligned}
  \]
  Because $\alpha I - Q(x^*) \succ 0$, if $W''_\alpha(x^*; d) = 0$,
  then $\zeta_d = 0$ and $d \in T_S(x^*)$.  But
  $T_\Sc(x^*) \cap \proxnormal_\Sc(x^*) = \{ 0 \}$, which means
  $W_\alpha''(x^*; d) > 0$ for all
  $d \in \proxnormal_\Sc(x^*) \setminus \{ 0 \}$, so by
  Lemma~\ref{lem:sufficient-condition}, $x^*$ is a weak sharp local
  minimizer of~$W_\alpha$.

  Next we verify condition (ii) in
  Lemma~\ref{lem:angle-convergence}. For all $x \in U$,
  \[
    \begin{aligned}
      \Dini_{{\Gc}_\alpha}W_\alpha(x) = -{\Gc}_\alpha(x)^\top(\alpha I - Q(x)){\Gc}_\alpha(x).
    \end{aligned}
  \]
  Without loss of generality, we can assume that $U$ is bounded. Then,
  we can choose $c_1, c_2 > 0$ so that
  \[
    \begin{aligned}
      c_1 &< \inf_{x \in U}\{ \lambda_\text{min}(\alpha
      I - Q(x)) \}
      \\
      c_2 &> \sup_{x \in U}\{ \lambda_\text{max}(\alpha I - Q(x)) \}.
    \end{aligned}
  \] 
  It follows that
  $\Dini_{{\Gc}_\alpha} W_\alpha(x) \leq -c_1 \norm{{\Gc}_\alpha(x)}^2$ for all
  $x \in U$, but since
  $\norm{\nabla W_\alpha(x)} \leq c_2 \norm{{\Gc}_\alpha(x)}$, we have for
  all $x \in U$,
  \[
    \Dini_{{\Gc}_\alpha}W_\alpha(x) \leq -\frac{c_1}{c_2} \norm{\nabla
      W_\alpha(x)} \norm{{\Gc}_\alpha(x)}.
  \]
  Finally, 
  we claim that $W_\alpha|_U$ is 
  a globally subanalytic function, and therefore 
  condition (i) holds by \cite[Thm. 1]{KK:98} and the fact 
  that the class of globally subanalytic sets is an o-minimal 
  structure (cf. \cite[Definition 1]{KK:98}). To prove the claim,
  first note that, since $f$ is real-analytic, $J_\alpha$ is
  real-analytic, and therefore subanalytic \cite[Definition
  3.1]{EB-PDM:88}. Since $U$ is bounded, and the restriction of any
  subanalytic function to a bounded open set is globally subanalytic~\cite{VDD-CML:96}, 
  it follows that $J_\alpha|_U$ is
  globally subanalytic. 
  Finally, {since $T^{(\alpha)}_\Cc|_U:U \rightrightarrows \real^n$ is a globally subanalytic set valued 
  map}, and
    \[ W_\alpha|_U(x) = \inf_{\xi \in T^{(\alpha)}_\Cc|_U(x)}\{
      J_\alpha|_U(x, \xi) \}, \] 
  {it follows by application of  
  Lemma~\ref{lem:partial} that $W_\alpha|_U$ is globally subanalytic. The statement follows by applying
   Lemma~\ref{lem:angle-convergence} with $\Kc = \real^n$.}
  \end{proof}

\subsection{Global Convergence}
Finally, we turn to the characterization of the global convergence
properties of the safe gradient flow. We show that when the problem
data are real-analytic and the feasible set is bounded, every trajectory
converges to a KKT point.

\begin{theorem}[Global convergence
    properties]\label{thm:global-convergence}
  Consider the optimization problem~\eqref{eq:opt}, and assume $\Cc$
  is bounded, $f$, $g$, and $h$ are real-analytic functions, and LICQ
  holds everywhere on $\Cc$. Let $X$ be an open set containing $\Cc$
  on which the safe gradient flow is well defined.  Then there is
  $\alpha^* > 0$ such that for $\alpha > \alpha^*$, every trajectory
  of the safe gradient flow starting in $X$ converges to some KKT
  point.
\end{theorem}

To prove Theorem \ref{thm:global-convergence}, we use the next result
characterizing the positive limit set of solutions of the safe
gradient flow.

\begin{lemma}[Convergence to connected
    component]\label{lem:connected-component}
  Consider the optimization problem~\eqref{eq:opt}, and assume $\Cc$
  is bounded, $f$, $g$, and $h$ are real-analytic functions, and MFCQ
  holds everywhere on $\Cc$. Let $X$ be an open set containing $\Cc$
  on which the safe gradient flow is well defined.  Then for all
  $x \in X$, $\omega(x)$ is contained in a unique connected component
  of~$\KKT$.
\end{lemma}
\begin{proof}
  By Theorem~\ref{thm:feasible-safety}, $\Cc$ is asymptotically stable
  and forward invariant on $X$, and by Lemma~\ref{lemma:decrescent},
  $\Dini_{\Gc_{\alpha}}f(x) \leq 0$ for all $x \in \Cc$. Using the
  terminology from~\cite{AA-CE:10}, $f$ is a \emph{height function} of
  the pair $(\Cc, \Gc_{\alpha})$.  
  
  Because $f, g,$ and $h$ are
  real-analytic and $\Cc$ is bounded, $\Cc$ is a globally subanalytic
  set. Let $\hat{f} = f + \delta_\Cc$.  Then $\hat{f}$ is a globally
  subanalytic function, $\hat{f}$ is continuous on
  $\dom(\hat{f}) = \Cc$, and $\KKT$ is precisely the set of critical
  points of $\hat{f}$. By the Morse-Sard Theorem for subanalytic
  functions \cite[Thm. 14]{JB-AD-AL:06}, $\KKT$ has at most a
  countable number of connected components, and $\hat{f}$ is constant
  on each connected component. Since $f(x) = \hat{f}(x)$ for all
  $x \in \Cc$, $f$ is also constant on each connected component of
  $\KKT$, meaning that the connected components of $\KKT$ are
  \emph{contained in $f$} (cf.  \cite[Definition 5]{AA-CE:10}).

  Hence, we can apply \cite[Thm. 6]{AA-CE:10}, and conclude that
  for all $x \in X$, the positive limit~set $\omega(x)$ is nonempty
  and contained in a unique connected component of $E$, where 
  \[E = \{x \in \Cc \mid \Dini_{{\Gc}_\alpha}f(x) = 0 \}.\]  
  However, by
  Lemma~\ref{lemma:decrescent}, $E = \KKT$, concluding the result.
  \end{proof}

We are ready to prove Theorem \ref{thm:global-convergence}.

\begin{proof}[Proof of Theorem \ref{thm:global-convergence}]
  By Lemma~\ref{lem:connected-component}, for $x \in X$, there is a
  connected component $\Sc \subset \KKT$ such that
  $\omega(x) \subset \Sc$.  Since LICQ holds on $\Sc$, by Proposition
  \ref{prop:strong-regularity} there is an open set $U$ containing
  $\Sc$ and Lipschitz functions
  $(u, v):U \to \real^{m}_{\geq 0} \times \real^{k}$ such that
  $U \cap \KKT = \Sc$ and $({\Gc}_\alpha(x), u(x), v(x))$ is the unique
  solution to \eqref{eq:KKT_proj} on $U$.
  
  Let $W_\alpha$ be given by~\eqref{eq:lyap-W}. By
  Lemma~\ref{lem:diff-W}, $W_\alpha$ is differentiable on $U$, and
  using the same reasoning as in the proof of
  Theorem~\ref{thm:non-isolated}, {$W_\alpha$ is a globally subanalytic
  function, and satisfies the Kurdyka-\Lojasiewicz inequality.}
  Furthermore, if
  $\alpha > \alpha^* = \sup_{x^* \in \Sc}\{\rho(Q(x^*)) \}$, then
  there is some $c > 0$ such that
  $\Dini_{{\Gc}_\alpha}W_\alpha(y) \leq -c \norm{\nabla
    W_\alpha(y)}\norm{{\Gc}_\alpha(y)}$ for all $y \in U$. 
    
  Thus, we can apply Lemma~\ref{lem:angle-convergence} with
  $\Kc = \real^n$ to conclude that every trajectory starting in
  $U$ that remains in $U$ for all time converges to a
  point in $\Sc$. However, since $\omega(x) \subset \Sc$, 
  there exists a $T > 0$ such that $\Phi_{T}(x) \in U$, 
  and for all $t > 0$,  
  $\Phi_t(\Phi_{T}(x)) = \Phi_{T + t}(x) \in U$. 
  Thus, there exists $x^* \in \Sc$ such that
  $\Phi_{T + t}(x) \to x^*$ as $t \to \infty$, and therefore
  the trajectory
  starting at $x$ converges to~$x^*$.
\end{proof}

{
  \begin{remark}[Lower bounds on the parameter $\alpha$ to ensure
      global convergence]
    {\rm Note that the proof of Theorem~\ref{thm:global-convergence}
      yields the expression
      $\alpha^* = \sup_{x^* \in \Sc} \{ \rho(Q(x^*)) \}$ for the lower
      bound on $\alpha$ that guarantees global convergence. In
      general, computing this expression requires knowledge of the
      primal and dual optimizers of the original problem. However,
      reasonable assumptions on $f$, $g$, and $h$ allow us to obtain
      upper bounds of $\alpha^*$. For instance, if $\Cc$ is polyhedral
      and $\nabla f$ is $\ell_f$-Lipschitz on $\Cc$,
      it follows that $\norm{\nabla^2 f(x)} \leq \ell_f$, and
      $\nabla^2 g_i(x) = 0$ and $\nabla^2 h_j(x) = 0$ for all
      $i=1, \dots m$ and $j=1, \dots k$. Therefore,
      $\alpha^* \leq \ell_f$, and $\ell_f$ can be used instead as a
      lower bound on $\alpha$ to ensure global convergence.  }
    \oprocend
\end{remark}
}

\section{Comparison With Other Optimization Methods}\label{sec:comparison}
Here we compare the safe gradient flow with other continuous-time
flows to solve optimization problems.  For simplicity, we restrict our
attention to an inequality constrained convex program.
Figure~\ref{fig:sim} shows the outcome of the comparison on the same
example problem taken from~\cite{AH-SB-GH-FD:21}.  The methods
compared are the projected gradient flow, the logarithmic barrier
method (see e.g.~\cite[Sec. 3]{AF-PEG-MHW:02}), the
$\ell^2$-penalty gradient flow (see e.g.~\cite[Ch. 4]{AVF-GPM:90}),
the projected saddle-point dynamics
(see e.g.,~\cite{AC-EM-SHL-JC:18-tac}), the globally projected
dynamics (see e.g.,~\cite{YSX-JW:00}), and the safe gradient flow.

Under the logarithmic barrier method, the feasible set is forward
invariant and the minimizer of the logarithmic barrier penalty
$f_{\text{barrier}}(x; \mu) = f(x) - \mu\sum_{i=1}^{m}\log(-g_i(x))$,
with $\mu>0$, does not correspond to the minimizer of \eqref{eq:opt}.
Under the unconstrained minimizer of the $\ell^2$-penalty,
$f_{\text{penalty}}(x; \epsilon) = f(x) +
\frac{\epsilon}{2}\sum_{i=1}^{m}[g_i(x)]^2_+ $, with $\epsilon > 0$,
does not correspond to the minimizer of \eqref{eq:opt}, and the
feasible set is not forward invariant under the gradient flow of
$f_{\text{penalty}}$. Under the projected saddle-point dynamics, the
feasible set is not forward invariant, but each trajectory converges
to $x^*$. Under the globally projected dynamics, the feasible set is
forward invariant, trajectories converge to~$x^*$, and trajectories
are smooth. However, unlike the safe gradient flow, the globally
projected dynamics may be undefined when the constraints are not convex.

\begin{figure}[t!]
  \centering%
  
  \subfigure[Projected gradient flow]{
    \includegraphics[width=.45\linewidth]{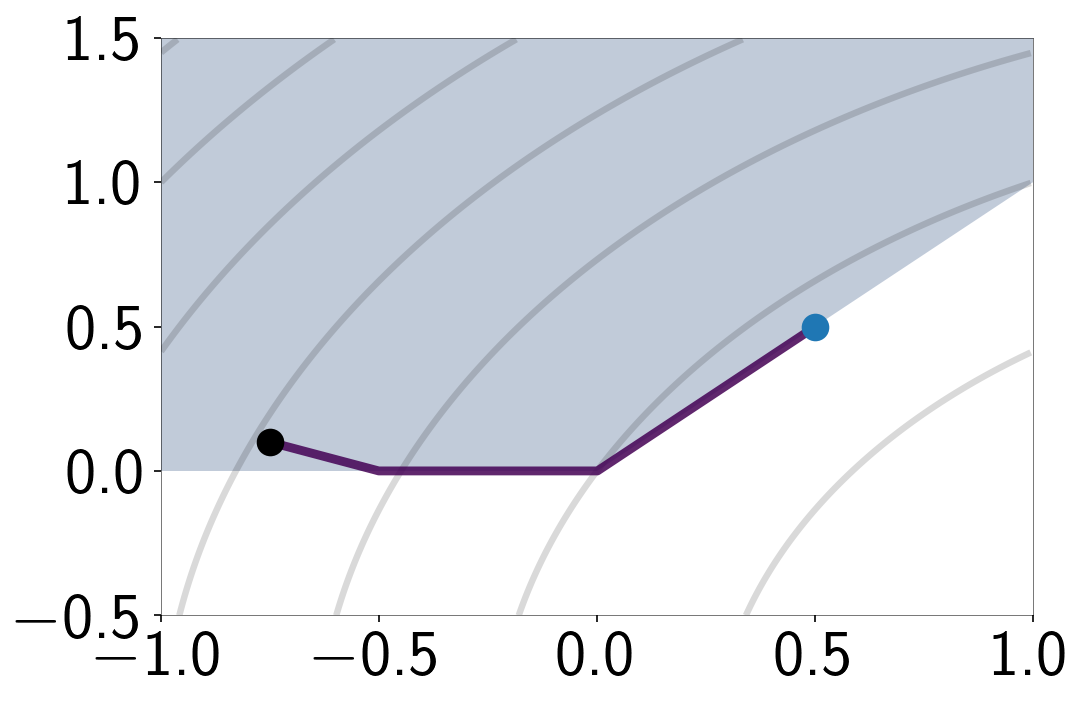}
    \label{subfig:pgf}
  }
  \subfigure[Logarithmic barrier flow]{
    \includegraphics[width=.45\linewidth]{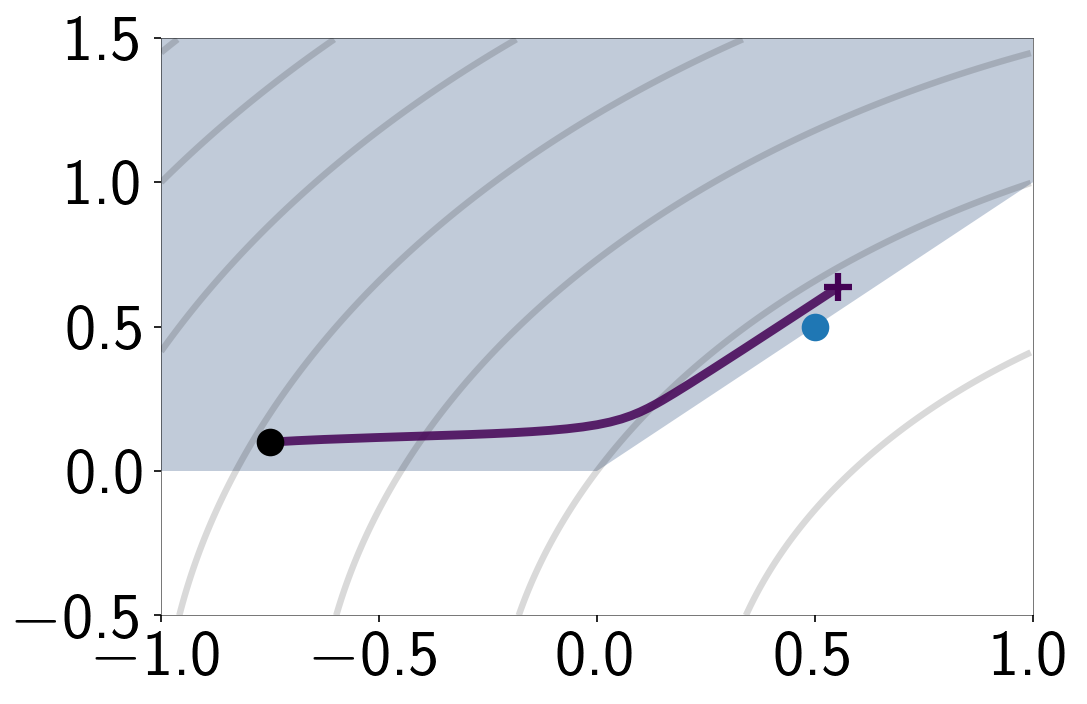}
    \label{subfig:log-barrier}
  }
  \\
  \subfigure[$\ell^2$-penalty gradient flow]{
    \includegraphics[width=.45\linewidth]{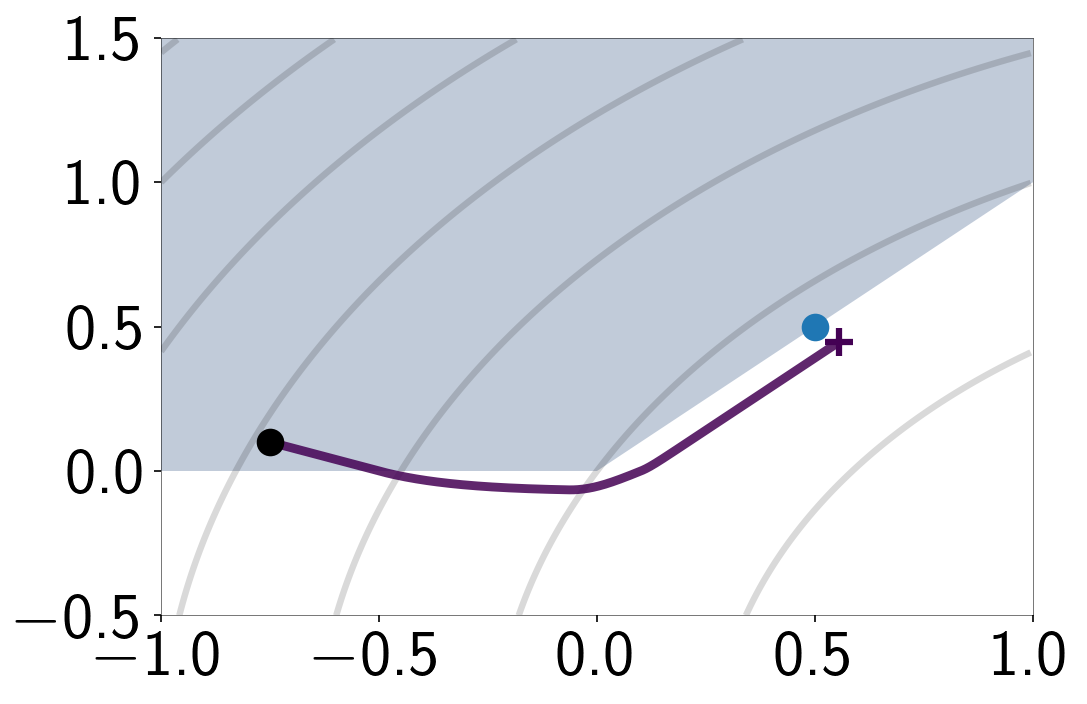}
    \label{subfig:l2-penalty}
  }
  \subfigure[Projected saddle-point dynamics]{
    \includegraphics[width=.45\linewidth]{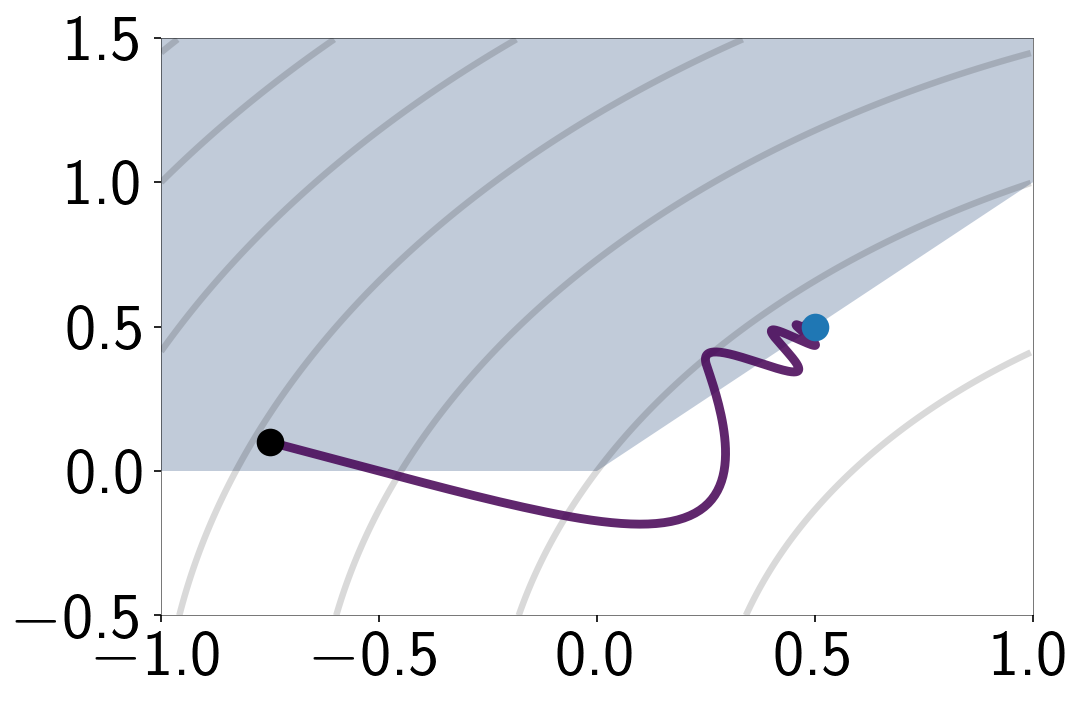}
    \label{subfig:psd}
  }
  \\
  \subfigure[Globally projected dynamics]{
    \includegraphics[width=.45\linewidth]{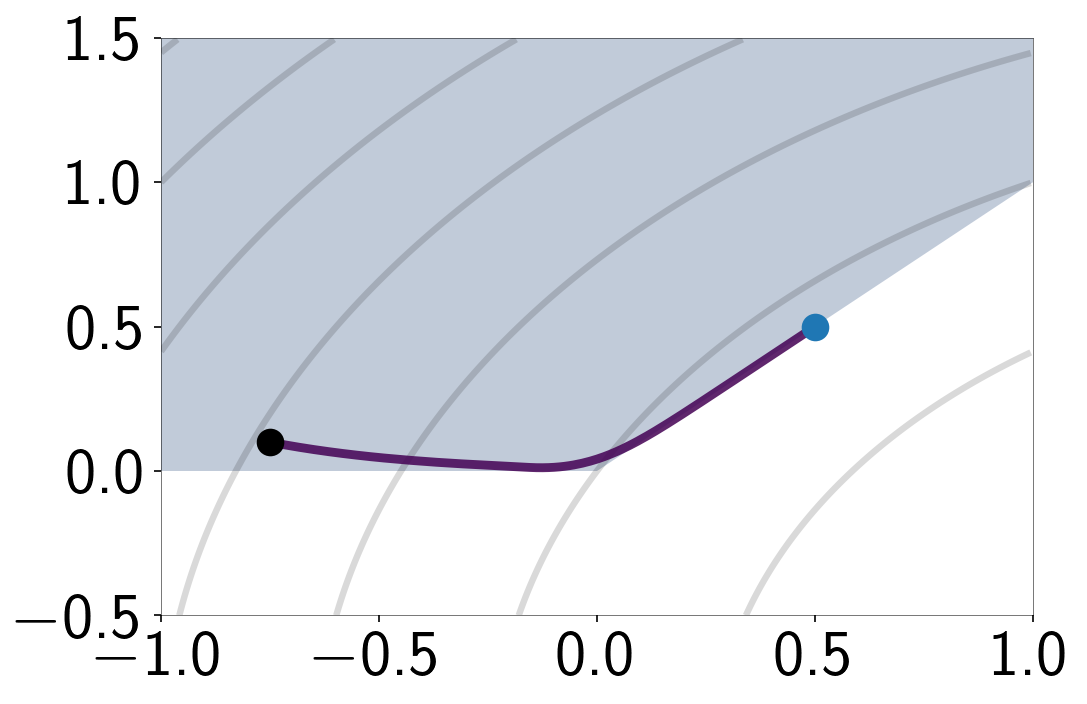}
    \label{subfig:globally-projected}
  }
  \subfigure[Safe gradient flow]{
    \includegraphics[width=.45\linewidth]{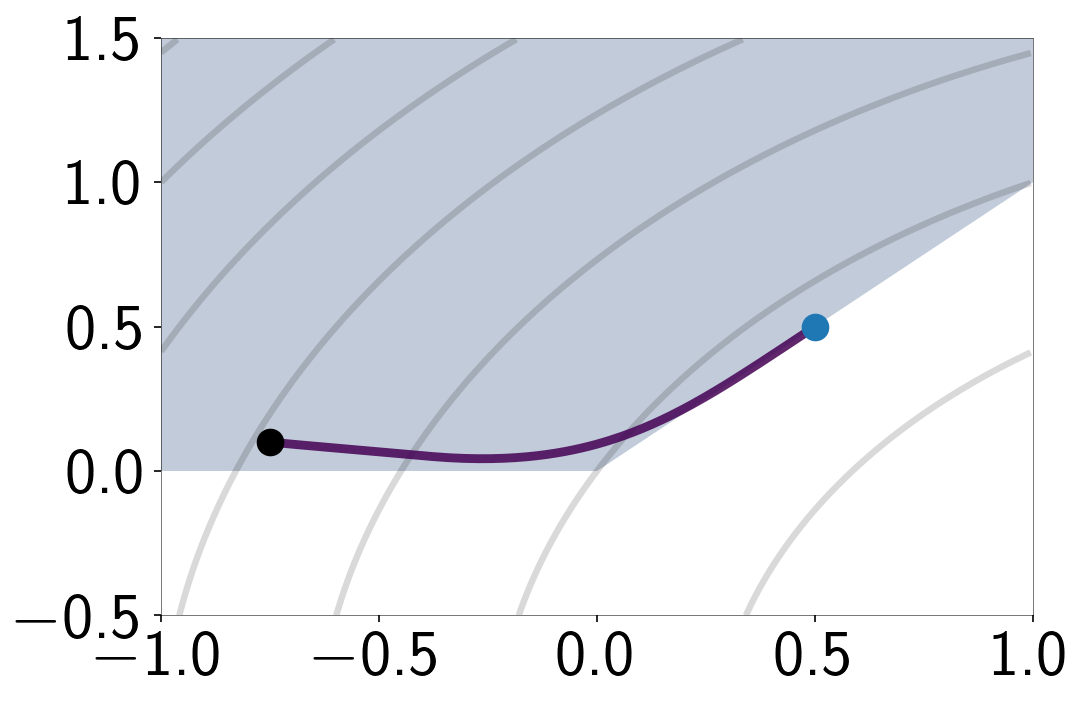}
    \label{subfig:safe-gradient-flow}
  }
  \caption{Comparison of methods minimizing
    $f(x) = 0.25\norm{x}^2 - 0.5x_1 + 0.25x_2$ subject to $x_2 \geq 0$
    and $x_1 \leq x_2$ (see also \cite[Figure 8]{AH-SB-GH-FD:21} for a
    comparison of additional methods).  The
    blue-shaded region is the feasible set and the grey curves are
    level sets of the objective function. The initial condition is
    denoted by the purple dot, and the global minimizer is denoted by
    a blue dot. (a) The trajectory converges to the global minimizer,
    and the trajectory remains inside the feasible set for all time
    but it is nonsmooth. (b) The trajectory is smooth and remains
    inside the feasible set but does not converge to the global
    minimizer. However, by choosing $\mu$ small enough, the trajectory
    can be made to converge arbitrarily close to the minimizer. (c)
    The trajectory is smooth, but does not remain inside the feasible
    set or converge to the global minimizer. However, by choosing
    $\epsilon$ small enough, the trajectory can be made to converge
    arbitrarily close to the minimizer.  (d) Initialized with
    $u(0) = 0$, the trajectory does not remain inside the feasible
    set, but it converges to the global minimum. (e) The trajectory is
    smooth, converges to the global minimizer, and remains inside the
    feasible set. However, this method may not be well-defined for
    nonconvex problems (f) The trajectory is smooth, converges to the
    global minimizer, and remains inside the feasible set.  
  }\label{fig:sim}
  \vspace*{-3ex}
\end{figure}

\section{Conclusions}\label{sec:conclusions}
We have introduced the safe gradient flow, a continuous-time dynamical
system to solve constrained optimization problems that makes the
feasible set forward invariant. The system can be derived either as a
continuous approximation of the projected gradient flow or by
augmenting the gradient flow of the objective function with inputs,
then using a control barrier function-based QP to ensure safety of the
feasible set.  The equilibria are exactly the critical points of the
optimization problem, and the steady-state inputs at the equilibria
correspond to the dual optimizers of the program.
We have conducted a thorough stability analysis of the dynamics,
identified conditions under which isolated local minimizers are
asymptotically stable and nonisolated local minimizers are semistable.
Future work will {explore the flow's robustness
  properties, and leverage convexity to obtain stronger global
  convergence guarantees.  Further, we hope to explore issues related
  to the practical implementation of the safe gradient flow, including
  interconnections of the optimizing dynamics with a physical system,
}, develop discretizations of the dynamics
and study their relationship with discrete-time iterative methods for
nonlinear programming, and {extend the framework to Newton-like
  flows for nonlinear programs which incorporate higher-order
  information.}

\appendices
\counterwithin{theorem}{section}

\def\thesection{\Alph{section}}%
\def\thesectiondis{\Alph{section}}%

\section{The Kurdyka-\Lojasiewicz
  Inequality}\label{subsec:o-minimal}
Here we discuss the Kurdyka-\Lojasiewicz inequality, which plays a
critical role in the stability analysis of the systems considered in
this paper.  The original formulation of the \Lojasiewicz
inequality\cite{SJ:83} states that for a real-analytic function
$V:\real^n \to \real$ and a critical point $x^* \in V^{-1}(0)$,
there exists  $\rho >0$, $\theta \in [0, 1)$, and $c > 0$ with
$ |V(x)|^\theta < c \norm{\nabla V(x)}$ for all~$x$ in a bounded
neighborhood of $x^*$ such that $ |V(x)| \leq \rho$. This inequality
is used to establish that trajectories of gradient flows
of real-analytic functions have finite 
arclength and converge pointwise to the set of equilibria.

In many applications, the assumption of real analyticity is too
strong. 
For example, the value function of a parametric nonlinear
program generally does not satisfy this assumption, even when all the
problem data is real-analytic.  
However, generalizations of the
\Lojasiewicz inequality have since been shown \cite{KK:98, JB-AD-AL:07}
to hold for much broader classes of functions, 
which can be characterized using the notion of o-minimal structures, 
which we define next.

\begin{definition}[o-minimal structures]
  For each $n \in \intpos$, let $\Oc_n$ be a collection of subsets of
  $\real^n$. We call $\{ \Oc_n \}_{n \in \intpos}$ an o-minimal
  structure if the following properties hold.
  \begin{enumerate}
    \item $\Oc_n$ is closed under complements, 
    finite unions and finite intersections.
    \item If $A \in \Oc_{n_1}$ and $B \in \Oc_{n_2}$ then $A \times B \in \Oc_{n_1 + n_2}$. 
    \item Let $\pi:\real^{n + 1} \to \real^{n}$ be the projection map onto the first $n$ 
    components. 
    If $A \in \Oc_{n + 1}$, then $\pi(A) \in \Oc_{n}$. 
    \item Let $g_1, \dots, g_m$ and $h_1, \dots, h_k$ be polynomial functions  
    on $\real^n$ with rational coefficients. 
    Then $\{ x \in \real^n \mid g_i(x) < 0, h_j(x) = 0, 
    1 \leq i \leq m, 1 \leq j \leq k \} \in \Oc_n$
    \item $\Oc_1$ is precisely the collection of all finite 
    unions of intervals in $\real$. 
  \end{enumerate}
\end{definition}

Examples of o-minimal structures include the class of semi-algebraic
sets and the class of globally subanalytic sets. We refer the reader
to~\cite{VDD-CML:96} for a detailed overview of these concepts. The
notion of o-minimality plays a crucial role in optimization theory,
since the remarkable geometric properties of definable functions
allows nonlinear programs involving them to be studied using powerful
tools from real algebraic geometry and variational analysis,
cf. \cite{ADI:09}.

Let $\{ \Oc_n \}_{n \in \intpos}$ be an o-minimal structure. A set $X
\subset \real^n$ such that $X \in \Oc_n$ is said to be \emph{definable
  with respect to $\{ \Oc_n \}_{n \in \intpos}$}. When the particular
o-minimal structure is obvious from context, then we simply call $X$
\emph{definable}. Given a definable set $X$ and $f:X \to \real^m$ and
$\Fc:X \rightrightarrows \real^m$, we say that $f$ (resp. $\Fc$) is
\emph{definable} if $\graph(f) \in \Oc_{n + m}$ (resp. $\graph(\Fc)
\in \Oc_{n + m}$).  The image and preimage of a definable set with
respect to a definable function is also definable, and the class of
definable functions is closed with respect to composition and linear
combinations.  Furthermore, as we show next, the value function of a
parametric nonlinear program is definable when the problem data is
definable.

\begin{lemma}[Definability of value functions]\label{lem:partial}
  Let $X\subset \real^n$, $J:X \times \real^{m} \to \real$ and $\Fc:X
  \rightrightarrows \real^{m}$ be definable. Let $V:X \to
  \extendedreal$ be given by $ V(x) = \inf_{\xi \in \Fc(x)}\{ J(x,
  \xi) \}$, and suppose that $\dom(V) = X$. Then $V$ is also
  definable.
\end{lemma}

Finally, functions definable on o-minimal structures satisfy a
generalization of the \Lojasiewicz inequality~\cite{KK:98}.

\begin{lemma}[Kurdyka-\Lojasiewicz inequality for
    definable functions]\label{lem:KL-inequality}
  Let $X \subset \real^n$ be a bounded, open, definable set, and
  $V:X \to \real$ a definable, differentiable function, and
  $V^* = \inf_{y \in X}V(y)$.  Then there exists $c > 0$, $\rho > 0$,
  and a strictly increasing, definable, differentiable function
  $\psi: [0, \infty) \to \real$ such that
  \[ \psi'(V(x) - V^*) \norm{ \nabla V(x)} \geq c \] for all $x \in U$
  where $V(x) - V^* \in  (0, \rho)$.
\end{lemma}

\section{Lyapunov Tests for Stability}
Here we present Lyapunov based tests for stability of an
equilibrium. The first result is a special case of \cite[Cor.
7.1]{SPB-DSB:03}, and establishes the stability of an isolated
equilibrium.

\begin{lemma}[Lyapunov test for relative stability]
  \label{lem:relative-stability}
  Let $\Kc$ be a forward invariant set of $\dot{x} = F(x)$ and $x^*$
  an isolated equilibrium. Let $U \subset \real^n$ be an open set
  containing $x^*$ and suppose that $V:U \cap \Kc \to \real$ is a
  directionally differentiable function such that
  \begin{itemize}
    \item $x^*$ is the unique minimizer of $V$ on $U \cap \Kc$.
    \item $\Dini_F V(x) < 0$ for all $x \in U \cap \Kc \setminus \{x^* \}$. 
  \end{itemize} 
  Then $x^*$ is asymptotically stable relative to $\Kc$. 
\end{lemma}

The next results provides a test for attractivity and stability of a
set of nonisolated equilibria, using an ``arclength''-based Lyapunov
test~\cite[Thm. 4.3 and Theorem 5.2]{SPB-DSB:10}.

\begin{lemma}[{{Arclength-based Lyapunov test}}]
  \label{lem:cont-stability}
  Let $\Kc$ 
  be a forward invariant set of $\dot{x} = F(x)$. Let $\Sc \subset \Kc$
  be a set of equilibria and $U \subset \real^n$ an open set
  containing $\Sc$ where $U \cap F^{-1}(\{0 \}) = \Sc$.
  Let $V:U \cap \Kc \to \real$ be a continuous
  function. Consider the following conditions.
  \begin{enumerate}
    \item There exists a $c > 0$ such that for all $x \in U \cap \Kc$,
    \begin{equation}
      \label{eq:arclength}
      \Dini_F V(x) \leq -c \norm{F(x)}.
    \end{equation}
    \item $x^*$ is a minimizer of $V$ if and only if $x^* \in \Sc$. 
  \end{enumerate}
  If (i) holds then every bounded trajectory that starts in $U \cap \Kc$ and 
  remains in $U \cap \Kc$ for all time has finite arclength and converges 
  to a point in $\Sc$. If (i) and (ii) hold
  then, in addition, every $x^* \in \Sc$ is semistable relative to $\Kc$.
\end{lemma}

In the case where the Lyapunov function $V$ is 
definable with respect to an o-minimal structure, 
we show that the condition in \eqref{eq:arclength} for the
arclength-based Lyapunov test can be replaced with
$ \Dini_F V(x) \leq -c \norm{F(x)}\norm{\nabla V(x)}$.  This is
referred to as the ``angle-condition'' and has been
exploited~\cite{PAB-RM-BA:05, CL:07-MN} to show convergence of descent
methods to solve nonlinear programming problems. The name arises from
the fact that the inequality implies that the angle between $F(x)$ and
$\nabla V(x)$ remains bounded in a neighborhood of the equilibrium.
In
the next result, we show that the angle condition, together with the
Kurdyka-\Lojasiewicz inequality, implies that all trajectories of the
system have finite arclength. 

\begin{lemma}[Angle-condition-based Lyapunov test]
  \label{lem:angle-convergence}
  Let $\Kc$ be a forward invariant set of $\dot{x} = F(x)$. Let
  $\Sc \subset \Kc$ be a bounded set of equilibria and
  $U \subset \real^n$ a bounded open set containing $\Sc$ where
  $U \cap F^{-1}(\{0 \}) = \Sc$.  Let $V:U \cap \Kc \to \real$ be a
  differentiable function. Consider the following conditions.
  \begin{enumerate}
    \item $V$ is constant and equal to $V^*$ on $\Sc$ and definable with respect to some o-minimal structure; 
    \item There is $c_2 > 0$ such that for all $x \in U \cap \Kc$,
    \[
      \Dini_F V(x) \leq -c_2 \norm{\nabla{V(x)}}\norm{F(x)}.
    \] 
    \item $x^*$ is a minimizer of $V$ if and only if $x^* \in \Sc$. 
  \end{enumerate}
  If (i) and (ii) hold then every trajectory that 
  starts in $U \cap \Kc$ and remains in 
  $U \cap \Kc$ for all time has finite arclength and converges to a point in $\Sc$. 
  If (i)-(iii) hold then, in addition, every $x^* \in \Sc$ is semistable relative to~$\Kc$.
\end{lemma}

\begin{proof}
  {
  Suppose (i) holds. By Lemma \ref{lem:KL-inequality}, 
  there exists $c_1 > 0$
  and a strictly increasing, definable, differentiable function
  $\psi: [0, \infty) \to \real$ such that 
  $\psi'(|V(x) - V^*|) \norm{ \nabla V(x)} \geq c_1$
  for all $x \in (U \cap \Kc) \setminus \Sc$.
  Assume without loss of generality that $\psi(0) = 0$, and 
  define $\tilde{V}:U \cap \Kc \to \real$
  by  
  \[ \begin{aligned}
    \tilde{V}(x) = \begin{cases}
      \psi(V(x) - V^*) \qquad &V(x) > V^* \\ 
      0 \qquad &V(x) = V^* \\ 
      -\psi(V^* - V(x)) \qquad &V(x) < V^*.
    \end{cases}
  \end{aligned} \]
  Then for all $x \in U$ with $V(x) > V^*$, we have
  \[ 
    \begin{aligned}
      \Dini_F\tilde{V}(x) &= \psi'(V(x) - V^*) \Dini_F V(x) \\
      &\leq -c_2 \psi'(V(x) - V^*) \norm{\nabla V(x)}\norm{F(x)} \\
      &\leq -c_1c_2 \norm{F(x)}.
    \end{aligned}
  \]
  A similar argument can be used to show that the previous inequality 
  also holds when $V(x) \leq V^*$. Since $\psi$ is increasing, 
  $x^* \in U \cap \Kc$ is a local minimizer
  of $\tilde{V}$ if and only if $x^*$ is a local minimizer of $V$. Hence, 
  the result 
  follows by applying
  Lemma \ref{lem:cont-stability} with the Lyapunov function $\tilde{V}$.  
  }
\end{proof}

\section{Regularity of Systems of Linear Inequalities}
The proof of Lemma \ref{lemma:feedback-feasibility}, requires the following 
technical result which gives conditions under which a linear system of 
inequalities is regular. 

\begin{lemma}
    \label{lem:double}
    Consider a linear inequality system in the variables $(u, v) \in \real^{m} \times \real^{k}$
    with the form 
    \begin{subequations}
        \label{eq:uv-system}
        \begin{align}
        \label{eq:uv-ineq}
        &G_1 u + G_2 v \leq c \\
        &H_1 u + H_2 v = h \\
        &u \geq 0
        \end{align}
    \end{subequations}
    where
    $c \in \real^{m}$, $d \in \real^{k}$, $G_1 \in \real^{m \times m}$,
    $G_2 \in \real^{m \times k}$, $H_1 \in \real^{k \times m}$, 
    $H_2 \in \real^{k \times k}$.
    The system~\eqref{eq:uv-system} is regular if
    \begin{enumerate}
      \item There exists $(u_0, v_0)$ satisfying \eqref{eq:uv-system} where 
      $G_1u_0 + G_2v_0 < c$;
      \item $H_2$ is full rank.
    \end{enumerate}
\end{lemma}
\begin{proof}
    By \cite[Theorem 2]{SMR:75}, the system \eqref{eq:uv-system} is
    regular if and only if there exists $(u_0, v_0)$ satisfying \eqref{eq:uv-system} where 
    $G_1u_0 + G_2v_0 < c$
    and the following system is regular: 
    \begin{subequations}
        \label{eq:reduced}
        \begin{align}
        H_1 u + H_2 v &= d \\ 
        u &\geq 0.
        \end{align} 
    \end{subequations}
    We claim that \eqref{eq:reduced} is regular whenever $H_2$ has
    full rank. Indeed, by a second application of \cite[Theorem
    2]{SMR:75}, \eqref{eq:reduced} is regular if and only if
    \begin{enumerate}
    \item[(a)] There exists $(u_1, v_1)$ with $u_1 > 0$ and
    $H_1u_1 + H_2v_1 = d$;
    \item[(b)] $[H_1 \;\; H_2]$ has full rank.
    \end{enumerate}
    If $H_2$ has full rank, (a) holds since for any $u_1 > 0$,
    we can always find some $v_1$ such that~$H_2 v_1 = h - H_1u_1$ 
    and (b) holds since if $H_2$ is
    full rank, $\text{range}([H_1 \;\; H_2]) = \text{range}(H_2)$.
\end{proof}

\section{Locally Bounded Set of Lagrange Multipliers}

The proof of Theorem~\ref{thm:stability}\ref{item:aymptotic-stability}
requires the following result, which establishes conditions under
which ${\Lambda_\alpha(x)}$ is locally bounded.

\begin{lemma}[Local boundedness of
    ${\Lambda_\alpha(x)}$]\label{lem:uniform-boundedness}
  Let $x^* \in \KKT$ and suppose MFCQ is satisfied at $x^*$. Let $U$
  be a bounded, open set containing $x^*$ on which
  \eqref{eq:smooth-proj-grad} is well defined and
  ${\Lambda_\alpha(x)} \neq \emptyset$ for all $x \in U$. Then,
  there exists $B < \infty$ with
  \begin{equation}
    \label{eq:uniform-bound}
    \sup_{x \in U} \left\{ \sup_{(u, v) \in \Lambda_\alpha(x)}
      \norm{ (u,v)}_\infty \right\} <
    B. 
  \end{equation}
  \vspace{0.1ex}
\end{lemma}

\begin{proof}
  By~\cite[Cor. 4.3]{SMR:82}, the solution map of
  \eqref{eq:smooth-proj-grad},
  $x \mapsto \{ \Gc_\alpha(x) \} \times \Lambda_\alpha(x)$, satisfies the
  Lipschitz stability property that there exists $\ell > 0$ where
  \begin{align}
    \label{eq:auxx}
    \norm{{\Gc}_\alpha(x)} + \dist_{\Lambda_\alpha(x^*)}(u, v) \leq \ell\norm{x
    - x^*},    
  \end{align}
  for all $(u, v) \in \Lambda_\alpha(x)$ and all $x \in U$.  By
  Proposition~\ref{prop:equilibria},~$\Lambda_\alpha(x^*)$ is precisely the
  set of Lagrange multipliers of~\eqref{eq:opt} at~$x^*$, so MFCQ
  implies that $\Lambda_\alpha(x^*)$ is bounded~\cite{DPB:99}.  Suppose by
  contradiction that \eqref{eq:uniform-bound} does not hold. Then
  there exists a sequence $\{ x^\nu \}_{\nu=1}^{\infty} \subset U$ and
  $(u^\nu, v^\nu) \in \Lambda_\alpha(x^\nu)$ where
  $\norm{(u^\nu, v^\nu)}_\infty \to \infty$
  as $\nu \to \infty$. Since $\Lambda_\alpha(x^*)$ is bounded,
  $\norm{{\Gc}_\alpha(x^\nu)} + \dist_{\Lambda_\alpha(x^*)}((u^\nu, v^\nu)) \to
  \infty$, which contradicts~\eqref{eq:auxx} and the fact that $U$ is
  bounded.
\end{proof}

\section{Jacobian of Safe Gradient Flow}
\label{ap:jacobian}
  The proof of Theorem \ref{thm:stability}(iii) relies on 
  analyzing the Jacobian of $\Gc_\alpha(x)$ at $x^*$. Here, we flesh out 
  the steps required to obtain the expression 
  for $\pfrac{\Gc_\alpha(x^*)}{x}$ in \eqref{eq:jac-G}. 
  Let $J = I_-(x^*)$ and assume, without loss of generality, that 
  the rows of $g(x^*)$ are ordered such that 
  $I_0 = \{1, 2, \dots, |I_0| \}$ and
  $J = \{|I_0| + 1, \dots, m \}$.  Let 
  $G = \pfrac{g(x^*)}{x}$, $G_I = \pfrac{g_{I_0}(x^*)}{x}$,
  $G_J = \pfrac{g_J(x^*)}{x}$ and $H = \pfrac{h(x^*)}{x}$.

  By the reasoning in the proof of Lemma \ref{lem:jacobian}, 
  strict complementarity and the strong second order sufficient 
  condition hold for the parametric optimization problem~\eqref{eq:smooth-proj-grad}. Therefore, 
  by~\cite[Theorem 2.3]{AVF:76} 
  it follows that the KKT triple\footnote{In this section we use the 
  convention that, in a KKT triple, the Lagrange multipliers corresponding 
  to the equality constraints are written before those corersponding to the inequality constraints.
  This ensures the Schur complement of $M$ in~\eqref{eq:M-partition} is block upper triangular, simplifying
  the computation of $M^{-1}$. }
  of \eqref{eq:smooth-proj-grad}, 
  written as $({\Gc}_\alpha(x),v(x), u(x))$, is differentiable 
  at~$x^*$ and the Jacobian $\Jc=\pfrac{}{x}({\Gc}_\alpha(x),v(x), u(x))|_{x=x^*}$ is 
  \begin{equation}
      \label{eq:Jdef}
      \begin{aligned}
          \Jc &= \begin{bmatrix}
              I & H^\top & G^\top \\ 
              -H & 0 & 0 \\ 
              -D_uG & 0 & -\alpha D_g
          \end{bmatrix}^{-1} \begin{bmatrix}
              -Q \\ \alpha H \\ \alpha D_u G 
          \end{bmatrix},
          \end{aligned} 
  \end{equation}
  where  
  $D_u = \textnormal{diag}(u^*)$ and $D_g = \textnormal{diag}(g(x^*))$. 
  By strict complementarity, 
  \[ \begin{aligned}
    &D_u = \begin{bmatrix}
      \tilde{D}_u & 0 \\ 0 & 0
    \end{bmatrix} \qquad &&D_g = \begin{bmatrix}
    0 & 0 \\ 0 & \tilde{D}_g
    \end{bmatrix} 
  \end{aligned}  \]
  where $\tilde{D}_u = \text{diag}(u^*_{I_0})$
  and $\tilde{D}_g = \text{diag}(g_J(x^*))$ are invertible matrices. Thus 
  the product \eqref{eq:Jdef} can be writen as 
  \[
    \label{eq:jacobian}
    \begin{aligned}
      &\Jc= 
    &\begin{bmatrix}
      I & H^\top & G_I^\top & G_J^\top \\ 
      -H & 0 & 0 & 0 \\ 
      -\tilde{D}_uG_I & 0 & 0 & 0 \\
      0 & 0 & 0 & -\alpha \tilde{D}_g
    \end{bmatrix}^{-1} \begin{bmatrix}
      -Q \\ \alpha H \\ \alpha \tilde{D}_u G_I \\ 0
    \end{bmatrix}.
    \end{aligned}
  \]
  Let $M$ be defined as  
  \begin{equation}
    \label{eq:M-partition}
    \begin{aligned}
      &M= 
    &\left[\begin{array}{c|ccc}
      I & H^\top & G_I^\top & G_J^\top \\ 
      \hline
      -H & 0 & 0 & 0 \\ 
      -\tilde{D}_uG_I & 0 & 0 & 0 \\
      0 & 0 & 0 & -\alpha \tilde{D}_g
    \end{array}\right].
    \end{aligned}
  \end{equation}
  Partitioning $M$ into a $2\times2$ block matrix as 
  in \eqref{eq:M-partition} allows us to compute its 
  inverse in closed form. 
  Let $N = M^{-1}$. Then by \cite[Theorem 2.1]{TTL-SHS:02},
  \[
    N = \begin{bmatrix}
      N_{11} & N_{12} \\ 
      N_{21} & N_{22}
    \end{bmatrix},
  \]
  where 
  \[
    \begin{aligned}
      N_{11} &= I - \begin{bmatrix}
        H \\ G_I
      \end{bmatrix}^\top\begin{bmatrix}
        HH^\top & HG_I^\top \\
          \tilde{D}_uG_IH^\top & \tilde{D}_uG_IG_I^\top
        \end{bmatrix}^{-1}\begin{bmatrix}
        H \\ \tilde{D}_uG_I
      \end{bmatrix}, \\
      N_{12} &= \begin{bmatrix}
        -\begin{bmatrix}
          H \\ G_I
        \end{bmatrix}^\top \begin{bmatrix}
          HH^\top & HG_I^\top \\
          \tilde{D}_uG_IH^\top & \tilde{D}_uG_IG_I^\top
        \end{bmatrix}^{-1} & \times
      \end{bmatrix}, \\
      N_{21} &= \begin{bmatrix}
        \begin{bmatrix}
          HH^\top & HG_I^\top \\
          \tilde{D}_uG_IH^\top & \tilde{D}_uG_IG_I^\top
        \end{bmatrix}^{-1} \begin{bmatrix}
          H \\ \tilde{D}_u G_I
        \end{bmatrix} \\ 0
      \end{bmatrix}, \\
      N_{22} &= \begin{bmatrix}\begin{bmatrix}
        HH^\top & HG_I^\top \\
        \tilde{D}_uG_IH^\top & \tilde{D}_uG_IG_I^\top
      \end{bmatrix}^{-1} & \begin{matrix}
        \hspace{1.1ex}\times \\ \hspace{1.1ex}\times 
      \end{matrix} \\
      \begin{matrix}
        0 \qquad\qquad & \qquad 0
      \end{matrix} & \times
    \end{bmatrix}, 
    \end{aligned} 
  \]
  where in the previous expressions 
  we replace values that will eventually be canceled out by $\times$.
  It follows that 
  \begin{subequations}
    \begin{align}
      \label{eq:jac-unsimplified}
      \pfrac{\Gc_\alpha(x^*)}{x} &= -N_{11}Q + \alpha N_{12}\begin{bmatrix}
        H \\ \tilde{D}_uG_I \\ 0
      \end{bmatrix}, \\  
      \begin{bmatrix}
        \pfrac{v(x^*)}{x} \vspace{0.5ex} \\ \pfrac{u(x^*)}{x}
      \end{bmatrix} &=  -N_{21}Q + \alpha N_{22}\begin{bmatrix}
        H \\ \tilde{D}_uG_I \\ 0
      \end{bmatrix}.    
    \end{align}
  \end{subequations}

  To simplify \eqref{eq:jac-unsimplified}, we start by letting $P$ be 
  the orthogonal projection onto $\ker \pfrac{g_{I}(x^*)}{x} \cap \ker \pfrac{h(x^*)}{x}$. 
  By \cite[Proposition 6.1.6]{DSB:05}, 
  \[ P = I -  \begin{bmatrix}
    H \\ G_I
  \end{bmatrix}^\dagger \begin{bmatrix}
    H \\ G_I
  \end{bmatrix}.
  \]
  Since LICQ holds at $x^*$, the matrix $[H; G_{I}]$ has full row rank, 
  so by \cite[Proposition 6.1.5]{DSB:05},
  \[ \begin{bmatrix}
    H \\ G_I
  \end{bmatrix}^\dagger = \begin{bmatrix}
    H \\ G_I
  \end{bmatrix}^\top\begin{bmatrix}
    HH^\top & HG_I^\top \\
    G_IH^\top & G_IG_I^\top
  \end{bmatrix}^{-1}.  \]
  Therefore, if we let $D = \textnormal{blkdiag}(I, \tilde{D}_u)$,
  then $D$ is invertible, and we have 
  \[ \begin{aligned}
    N_{11} &= I - \begin{bmatrix}
      H \\ G_I
    \end{bmatrix}^\top\left(D\begin{bmatrix}
    HH^\top & HG_I^\top \\
    G_IH^\top & G_IG_I^\top
  \end{bmatrix}\right)^{-1}D\begin{bmatrix}
    H \\ G_I
  \end{bmatrix} \\
  &= I -  \begin{bmatrix}
    H \\ G_I
  \end{bmatrix}^\top\begin{bmatrix}
    HH^\top & HG_I^\top \\
    G_IH^\top & G_IG_I^\top
  \end{bmatrix}^{-1}D^{-1}D\begin{bmatrix}
    H \\ G_I
  \end{bmatrix} \\
  &= P,
  \end{aligned} \]
  and, 
  \[
  \begin{aligned}
     N_{12}&\begin{bmatrix}
      H \\ \tilde{D}_uG_I \\ 0
    \end{bmatrix} \\
    &=-\begin{bmatrix}
      H \\ G_I
    \end{bmatrix}^\top \begin{bmatrix}
      HH^\top & HG_I^\top \\
      \tilde{D}_uG_IH^\top & \tilde{D}_uG_IG_I^\top
    \end{bmatrix}^{-1}\begin{bmatrix}
      H \\ \tilde{D}_uG_I
    \end{bmatrix} \\
    &= -\begin{bmatrix}
      H \\ G_I
    \end{bmatrix}^\top\left(D\begin{bmatrix}
    HH^\top & HG_I^\top \\
    G_IH^\top & G_IG_I^\top
  \end{bmatrix}\right)^{-1}D\begin{bmatrix}
    H \\ G_I
  \end{bmatrix} \\
  &= -(I - P).
  \end{aligned} 
  \]
  By substituting the previous two expressions into \eqref{eq:jac-unsimplified}, 
  we obtain
  \[ \pfrac{\Gc_\alpha(x^*)}{x} = -PQ -\alpha(I - P).  \]
\bibliographystyle{ieeetr}

\end{document}